\documentclass[12pt,a4paper]{amsart}
\usepackage{amsmath}
\usepackage{amsthm}
\usepackage{amssymb}
\usepackage{amsfonts}
\usepackage{mathtools}
\usepackage{xcolor}
\usepackage[utf8]{inputenc}
\usepackage{microtype, fullpage, wrapfig,textcomp,mathrsfs,csquotes,fbb, latexsym}
\usepackage[colorlinks=true, linkcolor=blue, citecolor=blue, urlcolor=blue, breaklinks=true]{hyperref}
\usepackage[capitalise]{cleveref}
\usepackage{todonotes}
\usepackage{longtable,tabularx}
\usepackage{graphicx}
\usepackage{float}
\usepackage{tikz}
\usepackage{xcolor}
\usepackage{rotating}
\usepackage{pst-node}
\usepackage{tikz-cd} 
\usetikzlibrary{arrows}
\theoremstyle{plain}
\newtheorem{theorem}{Theorem}
\newtheorem{corollary}{Corollary}[theorem]

\newtheorem{question}[]{Question}
\newtheorem{proposition}[theorem]{Proposition}
\numberwithin{theorem}{section}
\newtheorem{lemma}{Lemma}
\numberwithin{lemma}{section}
\newtheorem{example}{Example}
\numberwithin{example}{section}
\numberwithin{equation}{section}
\numberwithin{figure}{section}
\newtheorem{definition}{Definition}
\numberwithin{definition}{section}
\newtheorem{remark}{Remark}
\numberwithin{remark}{section}
\newcommand{\mbalpha}{\overline{\mathrm{M}}_{\alpha}}
\newcommand{\malpha}{\mathrm{M}_{\alpha}}
\newcommand{\kbalpha}{\overline{\mathrm{K}}_{\alpha}}
\newcommand{\kalpha}{\mathrm{K}_{\alpha}}
\newcommand{\abalpha}{\overline{\A}_{\alpha}}
\newcommand{\aalpha}{\A_{\alpha}}
\newcommand{\B}{\mathcal{B}}
\newcommand{\I}{\mathcal{I}}
\newcommand{\A}{\mathcal{A}}
\newcommand{\M}{\mathcal{M}}
\newcommand{\R}{\mathbb{R}}

\newcommand{\la}{\langle}
\newcommand{\ra}{\rangle}
\newcommand{\Z}{\mathbb{Z}}
\newcommand{\CS}{\mathrm{CS}}

\newcommand\preceqdot{\mathrel{\ooalign{$\preceq$\cr\cr
  \hidewidth\raise0.235ex\hbox{$\boldsymbol{\cdot}\mkern0.5mu$}\cr}}}
\newcommand{\precdot}{\prec\mathrel{\mkern-3mu}\mathrel{\boldsymbol{\cdot}}}
\title{Building planar polygon spaces from the projective braid arrangement}
\author[N\, Daundkar]{Navnath Daundkar}
\address{Indian Institute of Science Education and Research Pune, India.}
\email{navnath.daundkar@acads.iiserpune.ac.in}
\author[P\, Deshpande]{Priyavrat Deshpande}
\address{Chennai Mathematical Institute, India}
\email{pdeshpande@cmi.ac.in}
\thanks{PD is partially supported by a grant from the Infosys Foundation. 
This project is also supported by the MATRICS grant MTR/2017/000239 funded by SERB.}

\begin{document}
\maketitle

\vspace{-.8cm}
\begin{abstract}
The moduli space of planar polygons with generic side lengths is a smooth, closed manifold. 
It is known that these manifolds contain the moduli space of distinct points on the real projective line as an open dense subset. 
Kapranov showed that the real points of the Deligne-Mumford-Knudson compactification can be obtained from the projective Coxeter complex of type $A$ (equivalently, the projective braid arrangement)  by iteratively blowing up along the minimal building set. 
In this paper we show that these planar polygon spaces can also be obtained from the projective Coxeter complex of type $A$ by performing an iterative cellular surgery along a sub-collection of the minimal building set. 
Interestingly, this sub-collection is determined by the combinatorial data associated with the length vector called the genetic code. 
\end{abstract}

\section{Introduction}\label{intro}
A \emph{length vector} is a tuple of positive real numbers. 
The \emph{moduli space of planar polygons} associated with a length vector $\alpha=(\alpha_{1},\dots, \alpha_{m})$, denoted by $\malpha$, is the collection of all closed piecewise linear paths in the plane up to orientation preserving isometries with side lengths $\alpha_{1}, \alpha_{2},\dots, \alpha_{m}$. Equivalently, 
\[\malpha= \{(v_{1},v_{2},\dots,v_{m})\in (S^{1})^{m} : \displaystyle\sum_{i=1}^{m}\alpha_{i}v_{i} = 0 \}/\mathrm{SO}_{2},\]
where $S^{1}$ is the unit circle and the group of orientation preserving isometries $\mathrm{SO}_{2}$ acts diagonally.
The moduli space of planar polygons (associated with $\alpha$) viewed up to isometries is defined as 
\[\mbalpha\coloneqq \{(v_{1},v_{2},\dots,v_{m})\in (S^{1})^{m} : \displaystyle\sum_{i=1}^{m}\alpha_{i}v_{i} = 0 \}/\mathrm{O}_{2}.\]
A length vector $\alpha$ is called \emph{generic} if $\sum_{i=1}^{m}\pm \alpha_{i} \neq 0$. 
For such a length vector $\alpha$, the moduli spaces $\malpha$ and $\mbalpha$ are closed, smooth manifolds of dimension $m-3$. 
In the rest of this paper, the length vectors are assumed to be generic unless stated otherwise.

The manifold $\malpha$ admits an involution $\tau$ defined by
\begin{equation}\label{invo}
\tau(v_{1},v_{2},\dots,v_{m})=(\bar{v}_{1},\bar{v}_{2},\dots,\bar{v}_{m}),
\end{equation}
where $v_{i}=(x_{i},y_{i})$ and  $\bar{v}_{i}=(x_{i},-y_{i})$. 
Observe that $\tau$ maps a polygon to its reflected image across the $X$-axis. 
Since we are dealing with only generic length vectors, the involution $\tau$ does not have fixed points. 
It is clear that $\malpha$ is a double cover of $\mbalpha$. 

The moduli spaces of (planar) polygons have been studied extensively. For example, Farber and Schutz \cite{farber2007homology} proved that the integral homology groups of $\malpha$ are torsion-free. They also described the Betti numbers in terms of the combinatorial data associated with the length vector.
The mod-$2$ cohomology ring of $\mbalpha$ was computed by Hausmann and Knutson in \cite{HK1}. 

The configuration space of $m$-ordered, distinct points on  $\R P^1$ is 
\[C_{m}(\R P^1) := (\R P^1)^{m}\setminus \triangle,\] where $\triangle=\{(x_{1},\dots,x_{m})\in (\R P^1)^m : \exists ~ \hspace{.2mm} i, j \text{ such that } x_{i}=x_{j} \}.$
The \emph{real moduli space of genus zero curves} $\M_{0}^{m}(\R)$ is the quotient of $C_{m}(\R P^1)$ by $\mathbb{P}\mathrm{GL}_{2}(\R)$. 

There is a Deligne-Mumford-Knudson compactification $\overline{\M}_{0}^{m}(\R)$ of $\M_{0}^{m}(\R)$.
Kapranov \cite{Kapranov1993ThePM} showed that $\overline{\M}_{0}^{m}(\R)$ can be obtained from the projective Coxeter complex of type $A_{m-2}$ (equivalently the projective braid arrangement) by iteratively blowing up along the minimal building set. 
Moreover, this process results in a  regular cell structure on $\overline{\M}_{0}^{m}(\R)$ consisting of $(m-1)!/2$ copies of the associahedron as top-dimensional cells.  

It is known that for any generic length vector the corresponding planar polygon space is also a compactification of the real moduli space of genus zero curves. 
Also for every $m$, there is a unique class of length vectors for which $\malpha$ is the Coxeter complex and $\mbalpha$ is the projective Coxeter complex.
Therefore, it is natural to ask the following question. 

\begin{question}
Is there a way to obtain $\malpha$ (respectively $\mbalpha$) from the  Coxeter complex of type $A$ (respectively the projective Coxeter complex of type $A$) by some iterative topological operation?
\end{question}

In this article, we answer this question affirmatively. 
In order to achieve this, we introduce the notion of the (projective) cellular surgery on certain regular cell complexes (see \Cref{cs}).
The cellular surgery involves removing a subcomplex homeomorphic to the trivial tubular neighborhood of an embedded sphere by a complex which is a tubular neighborhood of a sphere of complementary dimension. 
The projective surgery is performed when there is a free $\Z_2$-action on the ambient space which also descends to tubular neighborhoods. 
For a given generic length vector $\alpha$ we first introduce a partial order on the collection of genetic codes, called the genetic order. 
We use this notion to describe the subcomplexes $\mathcal{G}_{\alpha}$ on which cellular surgery needs to be performed and subcomplexes $\mathbb{P}\mathcal{G}_{\alpha}$ on which the projective surgery is to be performed.
Interestingly, these complexes form a subcollection of the minimal building set.

Note that the collection of polygons in $\malpha$ (respectively $\mbalpha$) with exactly two parallel sides is a codimension-$1$ submanifold of $\malpha$ (respectively $\mbalpha$). It turns out that the collection  $\aalpha$  (respectively $\abalpha$) of such codimension-$1$ submanifolds of  $\malpha$ (respectively $\mbalpha$) forms a submanifold arrangement.
Consequently, there is a cell structure on both $\malpha$ and $\mbalpha$ induced by $\aalpha$ and $\abalpha$, respectively. 
These cell structures on $\malpha$ and $\mbalpha$, denoted by  $\kalpha$ and $\kbalpha$, respectively.
We prove the following theorem,
\begin{theorem}\label{mt}
Let $G$ be the genetic code of a length vector $\alpha=(\alpha_1,\dots,\alpha_m)$. Then
the iterated cellular surgery on the Coxeter complex $CA_{m-2}$ (respectively, on projective Coxeter complex $\mathbb{P}CA_{m-2}$) along the elements of $\mathcal{G}_{\alpha}$ (respectively $\mathbb{P}\mathcal{G}_{\alpha}$) produces the cell complex homotopy equivalent to $\kalpha$ (respectively $\kbalpha$).
\end{theorem}
This article is organised as follows:
In \Cref{pre}, we give the basics of braid arrangement, Coxeter complex, and planar polygon spaces. 
We also introduce the genetic order and derive some of its properties. 
We then introduce submanifold arrangements and study the induced cell structure. 
In  \Cref{Hausmann}, we give a description of Hausmann's theorem in the language developed in previous sections.
In  \Cref{cscpc}, we introduce the notion of cellular surgery on a simple cell complex and prove \Cref{mt}.

\section{The braid arrangement, Coxeter complex and motivation}\label{pre}
\subsection{The braid arrangement}
In this subsection, we set up notation and gather some results related to the Coxeter complex.

\begin{definition}
A finite collection, $\mathcal{A}$, of codimension-$1$ subspaces in the Euclidean space is called an arrangement of hyperplanes (or a hyperplane arrangement).
\end{definition}

\begin{definition}
The braid arrangement is the collection \[\B_m=\{H_{ij} : 1\leq i<j\leq m\},\]
where \[H_{ij}=\{(x_{1},\dots,x_{m})\in \R^{m} \mid  x_{i}-x_{j}=0\}.\] 
\end{definition}

An arrangement of hyperplanes is said to be \emph{essential} if the intersection of all hyperplanes is the origin.
The braid arrangement $\B_{m}$ is not essential, since \[\bigcap H_{ij}=\{(t,\dots,t)\in \R^{m} \mid t\in \R\}\neq \{0\}.\] 
Nevertheless, there is a way to make $\B_{m}$ essential by considering the quotient $\R^{m}/\cap H_{ij}$. 
Consider
\[V:=\{(x_{1},\dots,x_{m})\in \R^{m} \mid \sum_{i=1}^{m}x_{i}=0\},\]  
then it is easy to see that the collection \[\B_{V}=\{H_{ij}\cap V \mid 1\leq i<j\leq m\}\] is an essential arrangement in $V$. 
The arrangement $\B_{V}$ is called an \emph{essentialization} of $\B_{m}$.
Let $\mathbb{S}V$ be the unit sphere in $V$. 

\begin{definition}
The intersection of hyperplanes in $\B_{V}$ gives a simplicial decomposition of $\mathbb{S}V$.
This decomposition of $\mathbb{S}V$ is called the Coxeter complex of type $A_{m-1}$ and it is denoted by $CA_{m-1}$. 
The projective Coxeter complex $\mathbb{P}CA_{m-1}$ of type $A_{m-1}$ is the quotient of Coxeter complex $CA_{m-1}$ by the antipodal action. 
\end{definition}

It is clear that $CA_{m-1}$ has $m!$ copies of the $(m-2)$-simplex as its top-dimensional cells. 

\begin{example}
The Coxeter complex $CA_{3}$ is the $2$-dimensional sphere cellulated by $24$  triangles and $\mathbb{P}CA_{3}$ is the projective plane cellulated by $12$ triangles (see \Cref{fig:cox}).
\begin{figure}
    \centering
    \includegraphics[scale=.35]{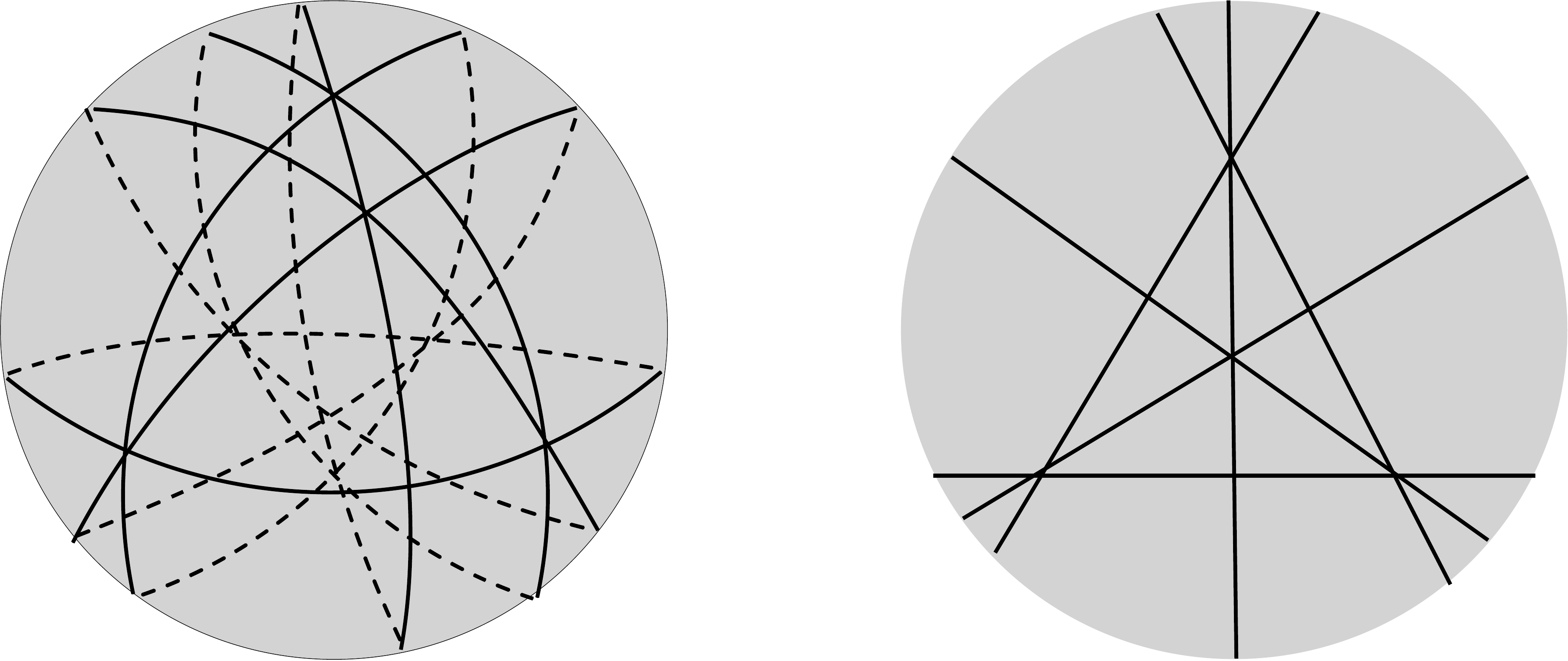}
    \vspace{2mm}
    \caption{The Coxeter complex $CA_{3}$ and the projective Coxeter complex $\mathbb{P}CA_{3}$}
    \label{fig:cox}
\end{figure}
\end{example}

The collection of all possible intersections of hyperplanes in the hyperplane arrangement $\mathcal{A}$ forms a lattice under reverse inclusion as the partial order. 
We denote this lattice by $\I(\A)$,
which is known as the  \emph{intersection lattice}.
Let $\I(\B_{m})$ be the intersection lattice of $\B_{m}$. 
It is clear that the lattices $\I(\B_{m})$ and $\I(\B_{V})$ are isomorphic. 
Moreover, it is isomorphic to the lattice of partitions of the set $[m]$, denoted by $\Pi_{m}$. 
If $\pi=J_{1}-\cdots-J_{k}$ is a partition of $[m]$ then one can associate to $\pi$ the following subspace: 
\[X_{\pi}=\{(x_{1},\dots,x_{m})\in V \mid x_{i}=x_{j} \text{ whenever $i$ and $j$ are in $J_{s}$ for some $1\leq s\leq k$} \}\] an element of $\I(\B_{V})$. 
The map \[\phi: \Pi_{m} \to I(\B_{V})\] defined by \[\phi(\pi) =X_{\pi}\] is an isomorphism.

De Concini and Procesi \cite{dconciniprocesi} identified a special collection of elements of the intersection lattice of an arrangement such that the blow-ups along these subspaces commute for a given dimension and the resulting arrangement has normal crossings.

For given intersection $X\in\I(\A)$ the subarrangement at $X$ is 
\[\A_{X} := \{H\in \A \mid X\subseteq H\}.\] 

\begin{definition}
An intersection $X\in \I(\A)$ is said to be reducible if there exist $Y$ and $Z$ in $\I(\A)$ such that $\A_{X}=\A_{Y}\sqcup \A_{Z}$, otherwise $X$ is irreducible. 
\end{definition}

\begin{definition}\label{mbs}
The minimal building set $\mathrm{Min}(\A)$ of $\A$ is the collection of all irreducible elements of $\I(\A)$.
\end{definition}

\begin{example}\label{min}
Consider the braid arrangement $\B_m$. 
The minimum building set $\mathrm{Min}(\B_{m})$ contains  intersections corresponding to those partitions of $[m]$ which have at most one block of size greater or equal  $2$.
\end{example}

Now we prove that for an element $X\in \I(\B_{V})$, the induced cell decomposition on the unit sphere in $X$ is a lower-dimensional Coxeter complex.

\begin{lemma}\label{icc}
Let $X\in \I(\B_{V})$ and $S_{X}=X\cap CA_{m-1}$. Then $S_{X}$ is isomorphic to the Coxeter complex $CA_{dim(X)-1}$.
\end{lemma}
\begin{proof}
Recall that $X=X_{\pi}$ for some partition  $\pi=(J_{1},\dots,J_{k})$ of $[m]$. 
Moreover, $\dim(X_{\pi})=k-2$. Note that $S_{X}$ is a sphere in $X$. We can think of the $k$ blocks of $X$ as the elements $\{1,2,\dots,k\}$. Then the induced cell structure on $S_{X}$ is equivalent to the cell structure on 
the unit sphere in $\R^k$ induced by the braid arrangement.
Therefore, $S_{X}\cong CA_{k-1}$. This proves the lemma. 
\end{proof}



\subsection{Motivation}\label{com}
Our article is motivated by the work of Hu \cite{stablepol} relating the Deligne-Knudson-Mumford compactification to the moduli space of spatial polygons, the work of Kapronov \cite{kapranov1993chow} expressing the aforementioned compactification as an iterative blow up and the work of Devadoss \cite{Devadoss} explaining the relationship over reals using combinatorial arguments. 
We briefly explain some of these ideas here. 

\begin{definition}
Let $\alpha=(\alpha_1,\dots,\alpha_m)$ be a generic length vector.
The spatial polygon spaces is defined as follows:
\[ \mathrm{N}_{\alpha}=\{(v_{1},v_{2},\dots,v_{m})\in (S^{2})^{m} : \displaystyle\sum_{i=1}^{m}\alpha_{i}v_{i} = 0 \}/\mathrm{SO}_{3},\]
where $\mathrm{SO}_{3}$ acts diagonally. 
\end{definition}
The spatial polygon spaces have been studied widely.
For example, the integer cohomology ring of $\mathrm{N}_{\alpha}$ was computed by Hausmann and Knutson in \cite{HK1}.  
 
The moduli space of $m$-punctured Riemann spheres (or the moduli space of genus zero curves) $\M_{0}^{m}$  is an important object in geometric invariant theory. There is the Deligne-Knudson-Mumford compactification $\overline{\M}_{0}^{m}$ of this space which has been studied widely. We refer the reader to  \cite{kapranov1993chow}, \cite{Kapranov1993ThePM}) for comprehensive introduction. 

In \cite{stablepol}, Hu introduced the notion of "stable polygons" (see \cite[Definition 4.13]{stablepol}). Roughly speaking, a \emph{stable} polygon is obtained from the following procedure: Let $P=(v_1,\dots,v_m)$ be a polygon and $J\subset [m]$ such that $v_i=v_j$ for $i,j\in J$. 
That is, sides of $P$ indexed by $J$ are parallel. 
Now introduce a new polygon without parallel edges, all whose sides except the longest one are indexed by $J$. 
The longest side is set to $\sum_{j\in J}\alpha_j-\epsilon$, where $\epsilon$ is a carefully chosen small positive real number. 
Denote this new polygon by $P_J$. 
Follow the same procedure for all sets of parallel sides and obtain such polygons without parallel edges. 
The stable polygon is a tuple of all such newly constructed polygons without parallel sides whose first coordinate is $P$.

Let $\mathcal{Y}$ be the collection of subvarities of $\mathrm{N}_{\alpha}$ defined in \cite[Section 6]{stablepol}.
The following theorem gives a relation between the moduli space of stable polygons $\mathfrak{M}_{\alpha,\epsilon}$, the Deligne-Knudson-Mumford compactification $\overline{\M}_{0}^{m}$ and spatial polygon space $\mathrm{N}_{\alpha}$.
\begin{theorem}[{\cite[Theorem 7.3, Theorem 6.5]{stablepol}}\label{HUTheorem}]
With the above notations
\begin{enumerate}
    \item The moduli space $\mathfrak{M}_{\alpha,\epsilon}$ is a complex manifold biholomorphic to $\overline{\M}_{0}^{m}$.
    \item The space $\mathfrak{M}_{\alpha,\epsilon}$ can be obtained from $\mathrm{N}_{\alpha}$, by iteratively blowing up along the elements of $\mathcal{Y}$.
\end{enumerate}

\end{theorem}



Recall the definition of configuration space $m$ ordered, distinct points of $\R P^1$ from the Introduction.

\begin{definition}
The real moduli space of $m$-punctured Riemann spheres is  
\[\M_{0}^{m}(\R)=\frac{C_{m}(\R P^{1})}{\mathbb{P}\mathrm{Gl}_{2}(\R)}.\]
\end{definition}

Let $\mathbb{P}(\B_{m-1})$ be the projective braid arrangement in $\mathbb{P}V$ and $\M(\mathbb{P}(\B_{m-1}))$ be its complement.
Let $(\mathbb{P}CA_{m-1})_{\#}$ denote the space obtained from $\mathbb{P}CA_{m-1}$ by iterated blow-ups along the minimal building set of $\mathbb{P}(\B_{m-1})$.

Let  $\overline{\M}_{0}^{m+1}(\R)$ be the real points of the Deligne-Mumford-Knudson compactification $\overline{\M}_{0}^{m+1}$.
Kapranov \cite{ Kapranov1993ThePM} remarkably proved the following .  
\begin{theorem}\label{kapra}
With the above notations
\begin{enumerate}
    \item There are homeomorphisms  $\M(\mathbb{P}(\B))\cong\M_{0}^{m+1}(\R)$ and  $\overline{\M}_{0}^{m+1}(\R)\cong (\mathbb{P}CA_{m-1})_{\#}$
    
    \item The real moduli $\overline{\M}_{0}^{m+1}(\R)$ is tiled by associahedra.
\end{enumerate}
\end{theorem} 


It is known that for a generic $\alpha$, the polygon space $\mbalpha$ contains $\M_{0}^{n}(\R)$ as an open dense set. 
In particular, $\mbalpha$ form a compactification of $\M_{0}^{n}(\R)$( see \cite{KL}, \cite{MJ} and \cite{pan} for more details).
Therefore, it is natural to ask the following questions.
\begin{question}
Is there a real version of \Cref{HUTheorem}?
\end{question}
\begin{question}
 Is there an analogue of \Cref{kapra} for planar polygon spaces?
\end{question}

By a theorem of Hausmann \cite[Proposition 2.9]{geohausmann} it follows that the planar polygon spaces are related by iterated surgery. 
Using a suitable cell structure (first defined by Panina \cite{pan}) we show that the surgery operation can be defined at the level of CW complexes and the genetic code of the given length vector can be used to keep track of how these spaces change.

\section{Planar polygon spaces}
This section is devoted to planar polygon spaces; we define genetic codes and prove some results.
In the end, we introduce a collection of codimension-$1$ submanifolds that form a submanifold arrangement and a induced cell structure on $\malpha$ and $\mbalpha$. 
This cell structure coincides with the one that Panina described in \cite{pan}.

We denote the set $ \{1,\dots, m\}$ by $[m]$. 
There are two important combinatorial objects associated with the length vector $\alpha=(\alpha_{1},\alpha_{2},\dots,\alpha_{m})$.
\begin{definition}
A subset $I\subset [m]$ is called $\alpha$-short  if 
\[\sum_{i\in I} \alpha_i  < \sum_{j \not \in I} \alpha_j\]  and $\alpha$-long otherwise.
\end{definition}
We may write short for $\alpha$-short when the context is clear. The collection of short subsets may be very large. There is another combinatorial object associated with length vectors which further compactifies the  the short subset data. 
Note that the diffeomorphism type of a planar polygon spaces does not depend on the ordering of the side lengths of polygons. Therefore, we assume that the length vector  satisfies $\alpha_{1}\leq \alpha_{2} \leq \dots\leq \alpha_{m}$. 

\begin{definition}\label{gc}
For a length vector $\alpha$, consider the collection of subsets of $[m]$ :
\[ S_{m}(\alpha) := \{J\subset [m] : \text{ $m\in J$ and $J$ is short}\} \]
and a partial order $\leq$ on $S_{m}(\alpha)$ by $I\leq J$ if $I=\{i_{1},\dots,i_{t}\}$ and $\{j_{i},\dots,j_{t}\}\subseteq J$ with $i_{s}\leq j_{s}$ for $1\leq s\leq t$. 
The \emph{genetic code} of $\alpha$ is the set of maximal elements of $S_{m}(\alpha)$ with respect to this partial order. 
\end{definition}
If $A_{1}, A_{2},\dots, A_{k}$  are the maximal elements of $S_{m}(\alpha)$ with respect to $\leq$ then the genetic code of $\alpha$ is denoted by $\la A_{1},\dots,A_{k}\ra$; we will use the notation $G(\alpha)$ when it is convenient to do so.

\begin{example}
Let $\alpha=(1,\dots,1,m-2)$ ($m$-tuple) be a length vector. Then the genetic code of $\alpha$ is $\la m\ra$. Moreover, $\malpha\cong S^{m-3}$ and $\mbalpha\cong \R P^{m-3}$.
\end{example}
For a generic length vector $\alpha$, consider the collection of all short subsets \[S(\alpha) :=\{J\subset [m] \mid J \text{ is $\alpha$-short}\}.\] 
\begin{theorem}[{\cite[Lemma 4.2]{HausmannRidriguez}}]
For a generic length vector $\alpha$,
the collection $S(\alpha)$ is determined by $S_m(\alpha)$. Moreover, $S(\alpha)$ can also be reconstructed from the genetic code of $\alpha$.
\end{theorem}

Observe that the partial order defined above doesn't depend on the length vector. In particular, this partial order remains a partial order on the set of all subsets of $[m]$ containing $m$. This fact will help us to introduce the partial order on the collection of genetic codes. 
\begin{definition}\label{genord}
Let $G(\alpha) = \la A_{1},\dots,A_{k}\ra$ and $G(\beta) = \la B_{1},\dots,B_{l}\ra$ be two genetic codes. We say that
\[\la A_{1},\dots,A_{k}\ra \preceq \la B_{1},\dots,B_{l}\ra\] if for each $1\leq i \leq k$ there exist $1\leq j\leq l$ such that  $A_{i}\leq B_{j}$. 
We call this partial order the genetic order.
\end{definition}

\begin{remark}\label{remgen}
Since $G(\alpha)$ consists of maximal elements of the poset $(S_m(\alpha),\leq)$, it completely determines the set $S_{m}(\alpha)$. 
Therefore, it follows from \cite[Lemma 1.2]{geohausmann} that if the genetic code of length vectors $\alpha$ and $\beta$ are same then the corresponding planar polygon spaces are diffeomorphic. 
Moreover, $G(\alpha)\preceq G(\beta)$ is equivalent to $S_m(\alpha) \subseteq S_m(\beta)$. 
Finally, in \Cref{genord} we may have $l\leq k$, for example, $\la 126,36\ra\preceq \la136\ra$. 
\end{remark}



\subsection{Saturated chains}
{Recall that in a poset $(P, \leq)$ an element $y$ is said to \emph{cover} another element $x$ if $x\leq z\leq y$ implies either $z=x$ or $z=y$. 
A \emph{saturated chain} is a totally ordered subset $C$ if there does not exist $z\in P\setminus C$ such that $x\leq z\leq y$ for some $x,y\in C$ and that $C\cup\{z\}$ is a chain.}

We now characterize saturated chains in the poset of genetic codes.
In particular, we first show that if the collection $S_{m}(\beta)$ is obtained by adding just one element to the collection $S_{m}(\alpha)$ then the code $G(\beta)$ covers the code $G(\alpha)$ in the genetic order.

\begin{proposition}\label{crgc}
The genetic code $G(\beta)$ covers $G(\alpha)$, {denoted $G(\alpha)\precdot G(\beta)$}, if and only if $S_{m}(\beta)=S_{m}(\alpha)\cup \{J\}$ for some $J\subset[m]$. 
\end{proposition}
\begin{proof}
Let $G(\beta)$ cover $G(\alpha)$.
Hence the collection $S_m(\alpha)$ is a subcollection of $S_m(\beta)$.
On the contrary, assume that there exists $J_1, J_2\in S_m(\beta)\setminus S_m(\alpha)$. 
{Let $G'$ denote the genetic code obtained by adding $J_1$ as a gene to $G(\alpha)$. }
Then we have $G(\alpha) \preceq G' \preceq G(\beta)$.
This is a contradiction to the fact that $G(\beta)$ covers $G(\alpha)$.
Therefore, we have $S_{m}(\beta)=S_{m}(\alpha)\cup \{J\}$ for some $J\subset[m]$.

Now we prove the converse by contradiction. 
Let there be a genetic code $H$ such that $G(\alpha)\preceq H\preceq G(\beta)$.
Then, there is a gene $I_1$ of $H$ which is not a part of $G(\alpha)$ and another gene $I_2$ of $G(\beta)$ which is not a part of $H$ (see \Cref{remgen}).
This implies, $I_1, I_2 \in S_m(\beta)\setminus S_m(\alpha)$, a contradiction. 
This concludes the proposition.
\end{proof}


\begin{remark}\label{shortalpha}
Recall from \cite[Lemma 4.1]{HausmannRidriguez} that the collection $S(\alpha)$ is determined by $S_m(\alpha)$.
Suppose $S_{m}(\beta)=S_{m}(\alpha)\cup \{J\}$ for some $J\subset[m]$.
Let $J'< J$ with $m\in J'$. Note that $J'\in S_{m}(\beta)$. Since $S_{m}(\beta)=S_{m}(\alpha)\cup \{J\}$, $J'\in S_{m}(\alpha)$. Consequently, $S_{m}(\alpha)$ generates all $\beta$-short subsets except $J$. Therefore, $S(\beta)= (S(\alpha)\setminus \{J^c\})\cup \{J\}$. 
\end{remark}

Now we express the covering relation at the level of genetic codes and then show how to construct saturated chains. 

\begin{proposition}
Let $\Tilde{G}$ and $G$ be two monogenic codes of the same size such that $S_m(\Tilde{G})=S_m(G)\cup J$ for some $J\subset [m]$. Then $\Tilde{G}=\la J\ra$.     
\end{proposition}
\begin{proof}
First we note that $J$ contains $m$.    
Suppose $G=\la S\ra$ where $S\subset [m]$ containing $m$. Then we have $G\leq \la S, J\ra \leq \Tilde{G}$.
\Cref{crgc} imples $\Tilde{G}$ covers  $G$. Therefore, $\Tilde{G}=\la S, J\ra$. But $\Tilde{G}$ is monogenic covering $G$. Thus $\Tilde{G}=\la J\ra$.
\end{proof}


\begin{proposition}\label{prop: monogen}
Let $G=\la \{g_1,\dots,g_r,m\}\ra$ and $\Tilde{G}=\la \{h_1,\dots,h_r,m\}\ra$ be two genetic codes. Then $\Tilde{G}$ covers $G$ if $h_i=g_1+i$ for $i\geq 1$ and $g_i=g_1+i$ for $i\geq 2$.    
\end{proposition}

\begin{proposition}\label{prop: gen covering}
Let $G=\la \{g_1,\dots,g_r,m\}\ra$ and $\Tilde{G}=\la \{h_1,\dots,h_{r+1},m\}\ra$ be two genetic codes. Then $\Tilde{G}$ covers $G$  if $h_i=i$ for $1\leq i\leq r+1$ and $g_j=j+1$ for $1\leq j\leq r$.
\end{proposition}

\begin{proposition}\label{prop: coveringrel}
 Let $G=\la \{g_1,\dots,g_r,m\}\ra$ be a genetic code and $G_i=\{g_1,\dots,g_{i-1},g_i-1,g_{i+1},\dots,g_r,m\}$ for $1\leq i\leq r$. Then $G$ covers the genetic code $\la G_1,\dots, G_r\ra$. 
\end{proposition}
\begin{proof}
The proof follows from the observation  $S_m(G)=S_m(\la G_1,\dots, G_r,\ra) \cup \{g_1,\dots,g_r,m\}$ and \Cref{crgc}.     
\end{proof}

\begin{remark}
We observe that the converse of \Cref{prop: monogen} and \Cref{prop: gen covering} is also true. Since we dont need the converse in the context of this paper, we opt not to write it here.    
\end{remark}

We now explain a procedure to construct a saturated chain of genetic codes. 
For simplicity we write the set $\{i_1,\dots,i_k\}$ as $i_1\dots i_k$.
Let $G=\la 25m\ra$. 
A genetic code covered by $G$ can be obtained using \Cref{prop: coveringrel}. 
For example, $G$ covers $\la 24m, 15m\ra$. 
We denote this genetic code by $C_1$. 
Then our first task is to obtain a saturated chain starting from $\la 24m \ra$ to $C_1$. 
Using \Cref{prop: coveringrel} for $\la 15m \ra$. In particular, we get the genetic code $C_2=\la 24m, 05m, 14m\ra=\la 24m,5m \ra$ which is covered by $C_1$. 
In the next step we reduce $C_2$ to $C_3=\la 24m,4m \ra=\la 24m\ra $. Thus we are done with the first task and the saturated chain is 

\begin{equation}\label{eq: satchain1}
C_3 = \la 24m \ra\precdot\la 24m, 5m \ra\precdot\la 24m, 15m \ra\precdot\la 25m \ra = G.    
\end{equation}

The next task is to keep using \Cref{prop: coveringrel} to reach from $C_3$ to $\la 23m \ra$. More precisely, we get this saturated chain as 
\begin{equation}\label{eq: satchain2}
\la 23m \ra\precdot\la 4m, 23m \ra\precdot\la 14m, 23m \ra\precdot C_3.    
\end{equation}
Now we can use \Cref{prop: monogen} to conclude that $\la 23m \ra$ covers $\la 13m \ra$. 
Then, we will reduce $\la 13m \ra$ to $\la 12m \ra$ using \Cref{prop: coveringrel}. Then we get the saturated chain 
\begin{equation}\label{eq: satchain3}
\la 12m\ra\precdot\la 3m,12m\ra\precdot\la 13m \ra. 
\end{equation}
Now there is an obvious way to reduce $\la 12m \ra$ to $\la m \ra$, which is 
\begin{equation}\label{eq: satchain4}
\la m \ra\precdot\la 1m \ra\precdot\la 2m \ra\precdot\la 12m \ra.    
\end{equation}
Finally, we use $\eqref{eq: satchain1}$, \eqref{eq: satchain2}, \eqref{eq: satchain3}, and \eqref{eq: satchain4} to obtain a saturated chain which starts with $\la m\ra$ and ends with $\la 25m \ra$ as follows 
\begin{align*}
\la m \ra\precdot\la 1m \ra\precdot\la 2m \ra\precdot\la 12m \ra\precdot\la 3m,12m\ra\precdot\la 13m \ra\precdot\la 23m \ra\precdot\la 4m, 23m \ra \\
\precdot\la 14m, 23m \ra\precdot\la 24m \ra\precdot\la 24m, 5m \ra\precdot\la 24m, 15m \ra\precdot\la 25m \ra.\\
\end{align*}
Given any monogenic code $G$, the general idea of constructing a saturated chain of genetic codes which starts with $\la m\ra$ and ends with $G$ is similar.  \
More precisely, start with a genetic code $G=\la\{ g_1,\dots, g_r,m\}\ra$. 
Then use \Cref{prop: coveringrel} repeatedly to get the covering chain that starts from $\la \{g_1,\dots,g_r-1,m\} \ra$ and reaches $G$ and then from $\la \{g_1,\dots,g_{r-1},g_{r-1}+1,m\} \ra$ to $\la \{g_1,\dots,g_r-1,m\} \ra$. 
One can use these ideas iteratively to construct a saturated chain from $\la \{g_1, g_1+2,\dots, g_1+r,m\}$ to $G$.
Then one can use \Cref{prop: monogen} to proceed further. 
Then again by using similar ideas to get a saturated chain starts from $\{1,2,\dots,r,m\}$ to   $\la \{g_1, g_1+2,\dots, g_1+r,m\}$. 
Then, there is an obvious saturated chain starts from $\la \{m\}\ra$ and ends with $\{1,2,\dots,r,m\}$.
Finally, we can combine all these saturated chains to get the saturated chain from $\la\{m\}\ra$ to $G=\la \{g_1,\dots g_r,m\}\ra$.

These ideas can be generalized to construct a saturated chain from $\la \{m\} \ra$ to any genetic code $G=\la G_1, \dots G_k\ra$ by constructing saturated chain between $\la G_i \ra$ and $\la\{m\}\ra$ for each $1\leq i\leq k$ and fixing all other genes in $G$.


\begin{remark}
    The answer to the question - would any arbitrary collection of subsets of $[m]$ that contain $m$ be a genetic code?- is no.  
    Since, in a genetic code any two genes can't be comparable and the complement of a gene can't be a short subset.
    For an arbitrary collection, any of the previous two conditions may fail. 
    However, it is possible to have a genetic code that does not correspond to any length vector. 
    Note that, we can define a short subset system abstractly as a collection of subsets of $[m]$ that satisfy: singletons are there, it is an abstract simplicial complex and if a subset is there in the collection then its complement is not there. 
    This definition gives rise to genetic codes that do not correspond to any length vector. 
    For example, $\la 2469 \ra$ is a genetic code that doesn't correspond to any length vector \cite[Lemma 4.6]{HausmannRidriguez}. 

    Another way of looking at (realizable) genetic codes is via hyperplane arrangements. 
    Equations of the form $\sum\pm x_i =0$ define a set of hyperplanes in $\R^m$. 
    Connected components of the complement of the union of these hyperplanes are called chambers. 
    An $m$-tuple in a chamber gives us a generic length vector. 
    It is not hard to check that changing the length vector in a chamber does not change the diffeomorphism type of the corresponding polygon spaces. 
    The same is true for genetic codes: there is a one-to-one correspondence between chambers of this arrangement in (the positive orthant of) $\R^m$ and (realizable) genetic codes in $[m]$.
    One can describe saturated chains of genetic codes using chambers of this arrangement. 
    Providing full details here will involve introducing more vocabulary and some technical results. 
    Since this language of arrangements and chambers is not directly relevant to the aim of this paper, we won't provide any details. 
    However, for the interested reader we only mention that the partial order on chambers is defined using `the set of separating hyperplanes' and the covering relation is given by `wall crossing'. 
\end{remark}

\subsection{Submanifold arrangements} 
There are smoothly embedded closed codimension-$1$ submanifolds of planar polygon spaces corresponding to $2$-element short subsets. In fact, the collection of all such submanifolds forms a submanifold arrangement. In this section, we study some combinatorial properties of this arrangement. We also study the induced cell structure.

\begin{definition}
Let $X$ be a finite dimensional smooth, closed manifold. 
A \emph{submanifold arrangement} is a finite collection $\A=\{N_{1},\dots,N_{r}\}$ of codimension-$1$ submanifolds such that,
\begin{enumerate}
    \item each element of $\A$ is smoothly embedded as a closed subset;
    \item for every point $x\in \cup_{i=1}^{r} N_{i}$ has a coordinate neighbourhood $V_{x}$ such that the collection $\{N_{1}\cap V_{x},\dots,N_{r}\cap V_{x}\}$ is a hyperplane arrangement in $V_{x}$ with $x$ as the origin;
    \item the intersections of members of $\A$ induces a regular cell structure o $X$ and each cell is combinatorially equivalent to simple convex polytope of an appropriate dimension.
\end{enumerate}
\end{definition}
There is an important combinatorial object associated with the submanifold arrangement. 
\begin{definition}
The intersection poset $\I(\A)$ is the set of connected components of all possible intersections of $N_i$’s ordered by reverse inclusion.
\end{definition}

Corresponding to every $2$-element short subset $\{i,j\}$ we have polygonal configurations with $i$-th and $j$-th  sides are in the same direction. Collection of such polygonal configurations forms codimension-$1$
submanifold of $\malpha$. 
In particular we write \[N_{i,j}=\{(v_{1},\dots,v_{i},\dots,v_{j},\dots,v_{m} )\in \malpha : v_{i}=v_{j}\}.\]
Let \[\alpha(i,j)=(\alpha_{1},\dots, \hat{\alpha_{i}},\dots,\hat{\alpha_{j}},\dots,\alpha_{i}+\alpha_{j},\dots,\alpha_{m})\] be the $(m-1)$-tuple such that $\alpha_{i}$ and $\alpha_{j}$ are absent in $\alpha(i,j)$. 
Observe that $\alpha(i,j)$ is a generic length vector. It is easy to see that $N_{i,j}\cong \mathrm{M}_{\alpha(i,j)}$. Similarly we define 
\[\overline{N}_{i,j}=\{(v_{1},\dots,v_{i},\dots,v_{j},\dots,v_{m} )\in \mbalpha : v_{i}=v_{j}\}.\] We also have $\overline{N}_{i,j}\cong \overline{\mathrm{M}}_{\alpha(i,j)}$. 
For a length vector $\alpha$, we define the finite collections of submanifolds of $\malpha$ and $\mbalpha$ as follows
\[\aalpha :=\{N_{i,j} : \{i,j\} \hspace{.1cm} \text{is} \hspace{.1cm} \alpha-\text{short}\},\] \[\abalpha :=\{\overline{N}_{i,j} : \{i,j\} \hspace{.1cm} \text{is} \hspace{.1cm} \alpha-\text{short}\}.\]

Let $\alpha$ be a generic length vector. Let \[\Pi_{m}(\alpha)=\{\pi\in \Pi_{m} : \textit{blocks of $\pi$ are $\alpha$-short}\}\] and \[\overline{\Pi}_{m}(\alpha)=\{\overline{\pi} : \pi\in \Pi_{m} \textit{ and
$\mathrm{M}_{\alpha(\pi)}$ is disconnected }\}.\]
Let $\mathcal{L}_{\alpha}= \Pi_{m}(\alpha) \sqcup \overline{\Pi}_{m}(\alpha)$ be the poset under the reverse refinement as a partial order. 

\begin{lemma}
The intersection posets $\I(\aalpha)$ and $\I(\mathcal{\overline{A}}_{\alpha})$ are isomorphic to the posets $\mathcal{L}_{\alpha}$ and $\Pi_{m}(\alpha)$, respectively.
\end{lemma}
\begin{proof}
Consider the following intersection \[X=N_{i_1j_1}\cap N_{i_2j_2}\cap \dots \cap N_{i_rj_r}.\] Then by clubbing together pairwise intersecting $2$-element short subsets \[\{i_l,j_l : 1\leq l \leq r\}\] we can write
\[X=N_{I_1}\cap N_{I_2}\cap\dots \cap N_{I_s},\] where $N_{I_t}=\bigcap_{\{i,j\}\subset I_t}N_{ij}$ For $1\leq t \leq s$. Note that $I_1-I_2-\dots-I_s$ is a partition of $\{i_1,j_1,\dots,i_r,j_r\}$. By putting together remaining singletons we get the partition of $[m]$. Let's denote this partition by $\pi$. Recall that if $X$ is disconnected then it is the disjoint union of tori. We label one of the connected component of $\pi$ and the other one by $\overline{\pi}$.
Otherwise, label $X$ by $\pi$.
Conversely, we define an element of $\I(\aalpha)$ corresponding to a  partition $\pi=J_1-\dots-J_k$ of $[m]$ with all $J_i$'s are short. 
Consider the following intersection.
\[X=(\cap_{\{i_{1},j_{1}\}\subset J_{1}}N_{i_{1}j_{1}})\cap \dots \cap (\cap_{\{i_{k},j_{k}\}\subset J_{k}}N_{i_{k}j_{k}}).\] As done above if $X$ is disconnected we label one of the connected component by $\pi$ and the other one by $\overline{\pi}$.

Note that if the intersection corresponding to $2$-element short subsets $\{i_l,j_l : 1\leq l \leq r\}$ \[\overline{X}=\overline{N}_{i_1j_1}\cap \overline{N}_{i_2j_2}\cap \dots \cap \overline{N}_{i_rj_r}\] is nonempty then $\overline{X}$ is connected. Now the isomorphism between $\I(\mathcal{\overline{A}}_{\alpha})$ and $\Pi_{m}(\alpha)$ is clear.
\end{proof}

\begin{remark}
Let $\alpha$ be a generic length vector.
and $\pi=J_1-\dots-J_k$ be a partition of $[m]$ with all $J_i$'s are $\alpha$-short.
Consider the shorter length vector $\alpha(\pi)=(\alpha_{J_{1}},\dots,\alpha_{J_{k}})$ where $\alpha_{J_{l}}= \sum_{i\in J_{l}}\alpha_{i}$ for $1\leq l \leq k$. 
Let \[X=(\cap_{\{i_{1},j_{1}\}\subset J_{1}}N_{i_{1}j_{1}})\cap \dots \cap (\cap_{\{i_{k},j_{k}\}\subset J_{k}}N_{i_{k}j_{k}})\] and \[\overline{X}=(\cap_{\{i_{1},j_{1}\}\subset J_{1}}\overline{N}_{i_{1}j_{1}})\cap \dots \cap (\cap_{\{i_{k},j_{k}\}\subset J_{k}}\overline{N}_{i_{k}j_{k}}).\] Then
it is easy to see that $X\cong \mathrm{M}_{\alpha(\pi)}$ and $\overline{X}\cong \overline{\mathrm{M}}_{\alpha(\pi)}$.
\end{remark}

\begin{corollary}\label{lcbraid}
Both the collections $\aalpha$ and $\abalpha$ are locally isomorphic to either braid arrangement or the product of braid arrangement.
\end{corollary}
\begin{proof}
Let $X \in \I(\aalpha)$ be a connected submanifold. Then without loss of generality assume that $X=J_1-J_2-\dots-J_k$, where $J_{i}$'s are  $\alpha$-short. Consider the collection \[\I({\A})_{X}=\{Y\in \I(\aalpha) : X\subseteq Y\}.\] Note that any element of $\mathcal{\A}_{X}$ has the labelled by the refined partition of $X$. Therefore, the poset $\I({\A})_{X}$ is isomorphic to the poset of all refinements of of partition $J_1-J_2-\dots-J_k$. This concludes \[\I({\A})_{X}\cong \prod_{i=1}^{k}\I(\B_{|J_i|}).\] Similar arguments work for $\abalpha$. 
\end{proof}

The following result is an immediate consequence of the above corollary. 
\begin{corollary}
The collections $\aalpha$ and 
$\abalpha$  induces a regular cell structure on $\malpha$ and $\mbalpha$, respectively such that each cell is combinatorially equivalent to some simple polytopes.
\end{corollary}

The following proposition is now clear.
\begin{proposition}
The collections $\aalpha$ and 
$\abalpha$  are submanifold arrangements in $\malpha$ and $\mbalpha$, respectively. 
\end{proposition}

We denote the cell structures induced from the submanifold arrangements $\aalpha$ and 
$\abalpha$ on $\malpha$ and $\mbalpha$ by $\kalpha$ and $\kbalpha$, respectively.


\begin{remark}
It can be observed that the cell structure $\mathrm{K}_{\alpha}$ induced by the submanifold arrangement coincides with the cell structure introduced by Panina in \cite{pan}. Panina also showed that for a generic length vector $\malpha$ is a PL-manifold. 
\end{remark}

\begin{example}\label{pl}
Let $\la m\ra$ be the genetic code of $\alpha$. Then we have  \[\abalpha=\{\overline{N}_{ij}: \{i,j\}\subset[m-1]\}.\] 
Note that for any proper subset of $[m-1]$ is $\alpha$-short if the genetic code of $\alpha$ is $\la m\ra$. Therefore, corresponding to any partition of $[m-1]$, we have nonempty intersection of $\overline{N}_{i,j}$'s. Therefore, it is easy to see that  \[\I({\overline{\A}_{\la m \ra}})\cong \Pi_{m-1}\setminus \{\hat{1}\},\] where $\Pi_{m-1}$ is the lattice of partitions of $[m-1]$. 
Note that $\malpha\cong S^{m-3}$ and the arrangement \[\aalpha=\{\overline{N}_{ij}: \{i,j\}\subset[m-1]\}\] is the braid arrangement $\B_{m-1}$ intersected with $S^{m-3}$.

\begin{figure}[H]
    \centering
    \includegraphics[scale=.35]{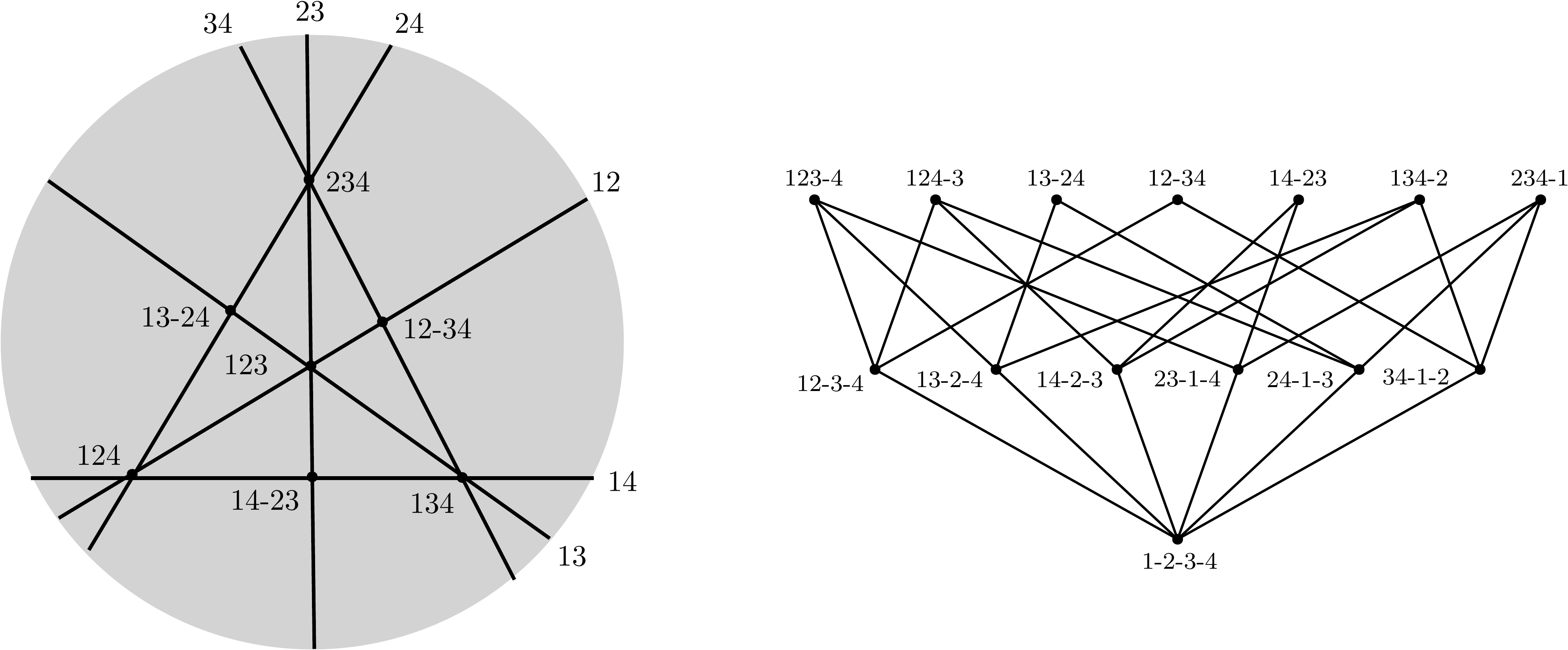}
    \vspace{.5cm}
    \caption{$\overline{\mathrm{K}}_{\la5\ra}\cong \mathbb{P}CA_{3}$ and $\I(\overline{\aalpha})\cong\Pi_{4}\setminus \{\hat{1}\}$}
    \label{fig: }
\end{figure}

\begin{figure}[H]
    \centering
    \includegraphics[scale=.26]{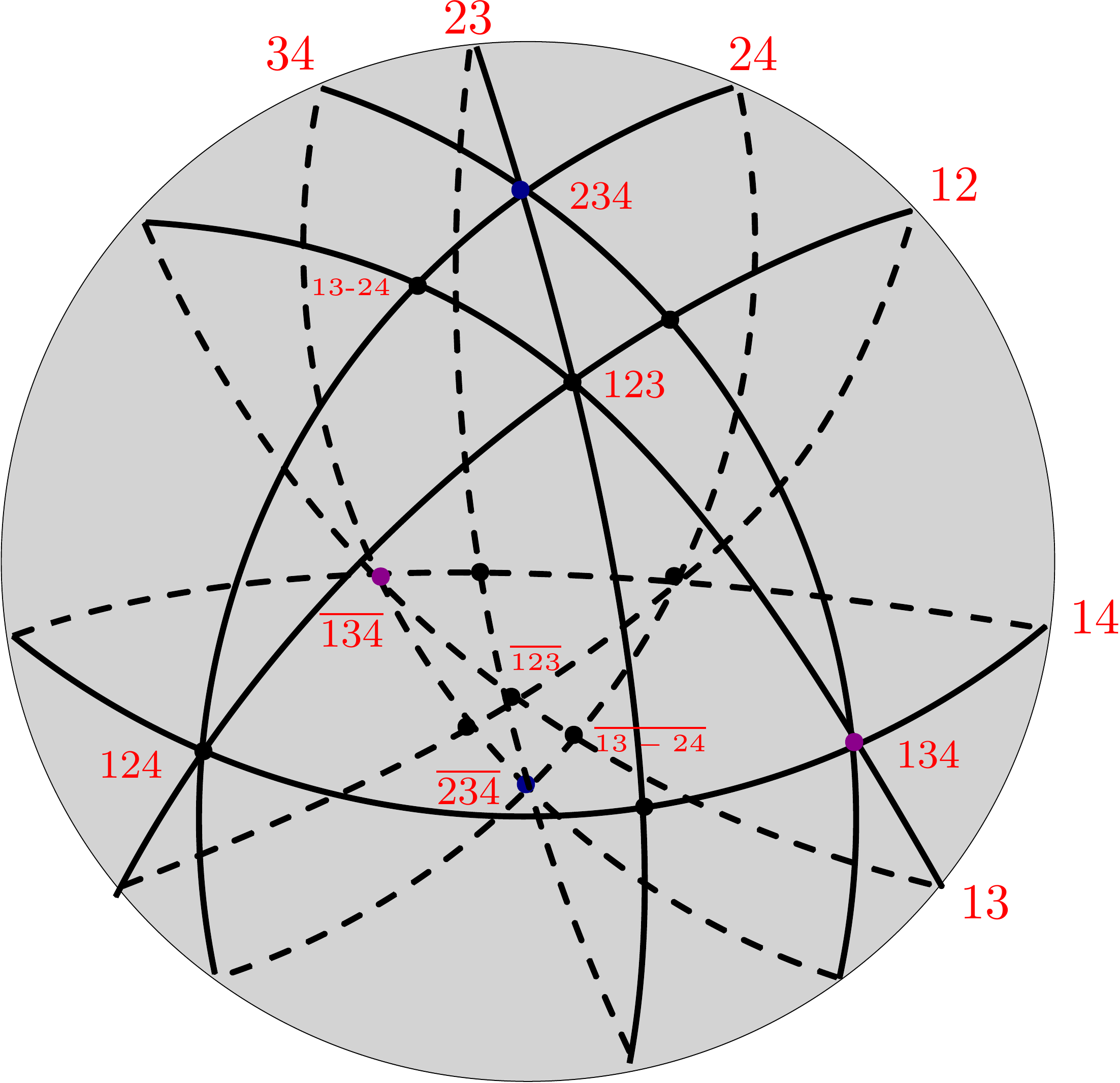}
    \vspace{.5cm}
    \caption{$\mathrm{K}_{\la5\ra}\cong CA_{3}$ with $\I(\mathrm{A}_{\alpha})$}
    
\end{figure}  

\end{example}

\begin{proposition}\label{Kcox}
The cell complex $\mathrm{K}_{\la m \ra}$ (respectively $\overline{\mathrm{K}}_{\la m \ra}$) is isomorphic to the Coxeter complex (respectively projective Coxeter complex) of type $A_{m-2}$.
\end{proposition}
\begin{proof}
Recall that $\mathrm{M}_{\la m \ra}\cong S^{m-3}$ and $\overline{\mathrm{M}}_{\la m \ra}\cong \R P^{m-3}$. 
Moreover, the submanifold arrangement $\A_{\la m \ra}$ is isomorphic to the braid arrangement $\B_{m-1}$; see \Cref{pl}. 
Therefore, it is evident that $\mathrm{K}_{\la m \ra}\cong CA_{m-2}$ and $\overline{\mathrm{K}}_{\la m \ra}\cong \mathbb{P}CA_{m-2}$.
\end{proof}

\section{A theorem of Hausmann}\label{Hausmann}
Let $\alpha$ and $\beta$ be two length vectors such that $S_{m}(\beta)= S_{m}(\alpha)\cup J$ for some $J\subset [m]$. Hausmann \cite{geohausmann} used techniques from Morse theory to obtain a relation between corresponding planar polygon spaces $\malpha$ and $\mathrm{M}_{\beta}$. He proved the following theorem.

\begin{theorem}[{\cite[Proposition 2.9]{geohausmann}\label{ht}}]
The space
$\mathrm{M}_{\beta}$ is obtained from $\malpha$ by an $O(1)$-equivarient surgery of index $|J|-2$. i.e.,
\[M_{\beta}\cong \left(M_{\alpha}\setminus S^{|J|-2}\times D^{m-1-|J|}\right)\underset{S^{|J|-2}\times S^{m-2-|J|}}{\bigcup}\left(D^{|J|-1}\times S^{m-2-|J|}\right) ,\]
where $O(1)$ acts antipodally on $D^{m-1-|J|}$ and $D^{|J|-1}$.
\end{theorem}

Note that using  \Cref{crgc}, we can say that if the genetic code $G$ covers $G'$ then $\mathrm{M}_{G}$ is obtained from $\mathrm{M}_{G'}$ by an $O(1)$-equivariant surgery. 
In fact, one can iterate this process over any saturated chain of genetic codes. 
Note that $\mathrm{M}_{\la m \ra}\cong S^{m-3}$ and $\overline{\mathrm{M}}_{\la m \ra}\cong \R P^{m-3}$.
The iterated version of \Cref{ht} is given by the following proposition.

\begin{proposition}
Let $\la m\ra=G_{1}\preceq G_{2}\preceq \dots \preceq G_{r}=G$ be the saturated chain of genetic codes. Then the space $\mathrm{M}_{G}$ is obtained from $S^{m-3}$ by an iterated $O(1)$-equivarient surgery.
\end{proposition}
\begin{proof}
Note that 
$S_{m}(G_{i+1})=S_{m}(G_{i})\cup J_{i}$ for $1\leq i \leq r-1$. Therefore, $\mathrm{M}_{G_{i+1}}$ is obtained from $\mathrm{M}_{G_{i}}$ by an $O(1)$-equivarient surgery along $S^{|J_{i}|-2}$. Observe that $S_{m}(G_{r})=\{m\}\displaystyle \cup_{i=1}^{r-1} J_{i}$. Now the propositions follows from iteratively applying  \Cref{ht}. 
\end{proof}

\begin{remark}
Observe that  \Cref{ht} doesn't describe how to keep track of iterations, also there is no CW-complex analogue of the procedure. 
\end{remark}

We now define a projective version of the surgery operation for certain quotient manifolds.
Let $M$ be a smooth manifold of dimension $n$ with a free $\Z_2$-action.
Suppose the $k$-dimensional sphere $S^k$ and its trivial tubular neighbourhood $S^k\times D^{n-k}$, embeds $\Z_2$-equivariantly in $M$. Let $\overline{M}$ denote the quotient of $M$ by the free $\Z_2$-action. Note that $\R P^k$ and the quotient $\frac{S^k\times D^{n-k}}{(x,y)\sim (-x,-y)}$ embed in $\overline{M}$.
With this information we introduce the following notations.
\begin{enumerate}
    \item \[\overline{D\mathbb{P}}(k):=\displaystyle \frac{S^k\times D^{n-k}}{(x,y)\sim (-x,-y)},\]
    
    \item \[\underline{D\mathbb{P}}(k):=\displaystyle \frac{D^k\times S^{n-k}}{(x,y)\sim (-x,-y)},\]
    
    \item \[\partial(\overline{D\mathbb{P}}(k))=\displaystyle\frac{S^k\times S^{n-k-1}}{(x,y)\sim(-x,-y)}=\partial(\underline{D\mathbb{P}}(k+1)).\]
\end{enumerate}




\begin{remark}
The space $\partial(\overline{D\mathbb{P}}(k))$ is the total space of the sphere bundle of the $(n-k)$-direct sum of canonical line bundles over $\R P^{k}$ and $\overline{D\mathbb{P}}(k)$ is the total space of the disc bundle of the $(n-k)$-direct sum of canonical line bundles over $\R P^{k}$.
\end{remark}

With the above notations, we define projective cellular surgery.
\begin{definition}
An \emph{index $k$-projective surgery} on a manifold $\overline{M}$ along $\R P^k$, produces a manifold $\mathbb{P}\mathrm{S}_k(\overline{M})$ defned as follows \[\mathbb{P}\mathrm{S}_{k}(\overline{M}):=\bigg(\overline{M}\setminus \overline{D\mathbb{P}}(k)\bigg)\bigcup_{\partial(\overline{D\mathbb{P}}(k))} \bigg(\underline{D\mathbb{P}}(k+1)\bigg).\]
\end{definition}

\begin{proposition}
We now have the following :
\begin{enumerate}
    \item The index-$0$ surgery on a manifold $M$ along $S^0$, produces a manifold homeomorphic to the connected sum $M\sharp (S^1\times S^{n-1})$.
    \item The index-$0$ projective surgery on a manifold $\overline{M}$ along $\R P^0$, produces a manifold homeomorphic to the connected sum $\overline{M}\sharp \R P^{n}$.
\end{enumerate}
\end{proposition}
\begin{proof}[Proof of (1)]
Without loss of generality $M=S^n$. 
Let $D^+$ and $D^-$ be two small and disjoint antipodal discs containing the north pole and south pole, respectively.
Then the surgery on $S^n$ along $S^0$ tells us that,  remove $D^+$ and $D^-$ from $S^n$ and attach $D^1\times S^{n-1}$ to $S^n\setminus(D^+\sqcup D^-)$. This clearly gives  $\mathrm{S}_0(S^n)=S^1\times S^{n-1}$.  
Observe that $S^n\sharp (S^1\times S^{n-1})=S^1\times S^{n-1}$.
Without loss of generality, we can assume that there is a bigger disc $D$ such that $D^+\sqcup D^-\subseteq D$
Now observe that the index-$0$ surgery on $S^n$ is an equivalent operation to removing $D$ from $S^n$ and attaching $(S^1\times S^{n-1})\setminus D'$ to $S^n\setminus D$, for some disc $D'$ in $S^1\times S^{n-1}$. This is same as the connected sum of $S^n$ and $S^1\times S^{n-1}$. 
The same idea works for general $M$.

\vspace{1mm}

\noindent \emph{Proof of (2).}
We make the following observations:
\begin{enumerate}
    \item \[\overline{D\mathbb{P}}(0)=\displaystyle \frac{S^0\times D^{n}}{(x,y)\sim (-x,-y)}=D^n,\] 
    \item \[\underline{D\mathbb{P}}(1)=\displaystyle \frac{D^1\times S^{n-1}}{(x,y)\sim (-x,-y)}=\displaystyle\frac{S^n\setminus (D^n_{+}\sqcup D^n_{-})}{x\sim -x}= \R P^n\setminus D^n,\]
    \item \[\partial(\overline{D\mathbb{P}}(0))=\displaystyle\frac{S^0\times S^{n-1}}{(x,y)\sim(-x,-y)}=S^{n-1}=\partial(\underline{D\mathbb{P}}(1)).\]
\end{enumerate}
Therefore, \[\mathbb{P}\mathrm{S}_{0}(\overline{M}):=\bigg(\overline{M}\setminus D^n\bigg)\bigcup_{S^{n-1}} \bigg(\R P^n\setminus D^n \bigg)=\overline{M}\sharp \R P^{n}.\]
This proves the result.
\end{proof}

\begin{theorem}[{\cite[Proposition 2.9]{geohausmann}}]\label{gdscr}
If the genetic code $G$ covers $G'$, i.e., $S_{m}(G)=S_{m}(G')\cup J$ for some $J\subset[m]$ then  $\overline{\mathrm{M}}_{G}$  is homeomorphic to $\mathbb{P}\mathrm{S}_{|J|-2}(\overline{\mathrm{M}}_{G'})$.
\end{theorem}

One can iterate the projective surgery to any chain $G_{1}\preceq G_{2}\preceq \dots \preceq G_{r}=G$ such that for each $1\leq i\leq r-1$, $G_{i}$ is covered by $G_{i+1}$. We denote the space after iterated projective surgery as  $\mathbb{P}\mathrm{S}_{(j_{1},\dots,j_{r})}(M_{G_{1}})$ where $j_{i}=|J_{i}|-2$ such that $S_{m}(G_{i+1})=S_{m}(G_{i})\cup J_{i}$. In fact we have $S_{m}(G_{r})=S_{m}(G_{1})\displaystyle \cup_{i=1}^{r-1} J_{i}$. 
With this, we have the following version of \Cref{ht}.
\begin{proposition}
The planar polygon space $\overline{\mathrm{M}}_{G}$  is homeomorphic to $\mathbb{P}\mathrm{S}_{(j_{1},\dots,j_{r})}(\R P^{m-3})$.
\end{proposition}

\section{Combinatorial surgery on a meet semi-lattice}

The notion of combinatorial blow-up was introduced by Feichtner and Kozlov in \cite{FeichtnerKozlov}.
Here, we introduce a similar notion in the contexts of surgery.

\begin{definition}
Let $\mathcal{L}$ be a meet semilattice. 
For an element $x\in \mathcal{L}$, we define a poset $\CS_x(\mathcal{L})$, the combinatorial surgery on $\mathcal{L}$ along $x$, as follows:

\begin{itemize}
    \item elements of $\CS_x(\mathcal{L})$:
    \begin{enumerate}
    \item $y\in \mathcal{L}$, $y\neq x$ and $y\ngeq x$
    \item $[x,y]$ , $y<x$
\end{enumerate}
\item order relations in $\CS_{x}(\mathcal{L})$:
\begin{enumerate}
    \item $y>z$ in $\CS_{x}(\mathcal{L})$ if $y>z$ in $\mathcal{L}$
    \item $[x,y]>[x,z]$ in $\CS_{x}(\mathcal{L})$ if $y>z$ in $\mathcal{L}$
    \item $[x,y]>z$ in $\CS_{x}(\mathcal{L})$ if $y\geq z$ in $\mathcal{L}$.
    \item $y<[x,\hat{0}]$ if $y\vee x\in \mathcal{L}$.
\end{enumerate}
\end{itemize}
\end{definition}

\begin{remark}
The element $[x,\hat{0}]$ can be thought of as a result of combinatorial surgery along $x$.
\end{remark}

\begin{theorem}
The poset $\CS_x(\mathcal{L})$ is a meet semilattice. Moreover, for $x\in \mathcal{L}$,
the posets $\mathcal{L}$ and $\CS_x(\mathcal{L})$ are of equal rank
$k$ the rank of $\mathcal{L}$,
then \[\mathrm{rk}([x,y])=k-\mathrm{rk}(x)+\mathrm{rk}(y)+1.\]

\end{theorem}

\begin{example}
Let $G=\la\{2,6\}\ra$ be the genetic code and $\I(\A_{G})$ be the corresponding meet semilattice. 
Let $(1,2,345,6)\in \I(\A_{G})$. 
We denote this partition by $345$.
Then 

\begin{equation*}
\begin{split}
    \CS_{345}(\I(\A_G))& =\bigg(\I(\A_G)\setminus \I(\A_G)_{\geq 345} \bigg)~\bigsqcup~ \bigg\{[345,y] : y<345 \bigg\}\\
    & \cong ~\bigg(\I(\A_G)\setminus \I(\A_G)_{\geq 345} \bigg)~\bigsqcup~ \bigg\{(126,\pi) : \pi<(1,,2,345,6)\bigg\}\\
    & = \I(\A_{\la\{1,2,6\}\ra}),
\end{split} 
\end{equation*}
where $(126,\pi)$ denotes an unordered partition of $[6]$.
Observe that the genetic code $\la\{1,2,6\}\ra$ covers $\la\{2,6\}\ra$ with respect to the genetic order.

\begin{figure}[H]
    \centering
    \includegraphics[scale=0.4]{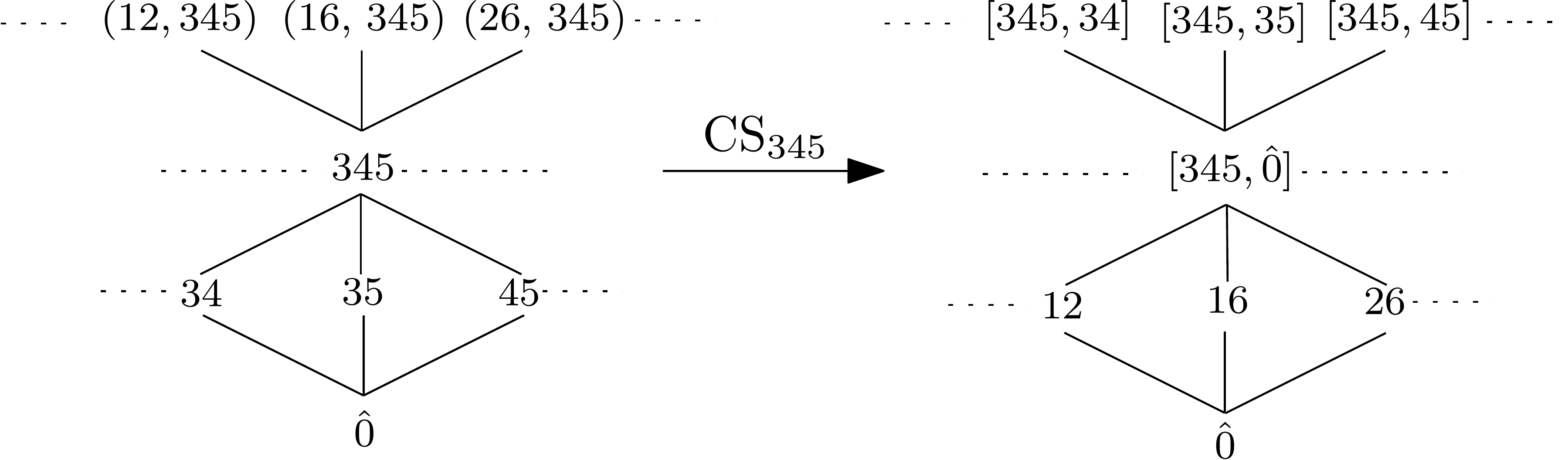}
    \caption{Combinatorial surgery along $345$}
    \label{}
\end{figure}
\end{example}
 
Let $G$ and $G'$ be two genetic codes of $m$-length vectors such that $G'$ covers $G$. 
It follows from \Cref{crgc} that there exists a subset $J\subseteq [m]$ with $S_m(G')=S_m(G)\cup J$.  
With this, now the following result is straightforward.
\begin{proposition}
$\CS_{J^c}(\I(\A_{G}))\cong \I(\A_{G'}) $.
\end{proposition}

\section{Cellular surgery on simple cell complexes and the proof of  \texorpdfstring{\Cref{mt}}{Lg}}\label{cscpc}
Let $K$ be a simple cell complex of dimension $n$ such that there is a subcomplex homeomorphic to the $k$-sphere $S^k$.
Let us denote this subcomplex by $KS^k$. 
Moreover, assume that for any $k$-simplices $\sigma, \sigma'\in KS^k$, $\mathrm{Lk}(\sigma,K)\cong\mathrm{Lk}(\sigma',K)\cong S^{n-k-1}$. 
\begin{definition}
The index $k$ cellular surgery on $K$ along $KS^k$ is defined in two steps:
\begin{enumerate}
    \item[Step 1:]  Truncate all cells whose closure intersects $KS^k$.
    \item[Step 2:] Let $D(KS^k)$ be the cellular disc with the boundary $KS^k$. Note that the boundary complex of the truncated part around $KS^k$ is $KS^k\times \mathrm{Lk}(\sigma,K)$ for $\sigma\in KS^k$. Now attach another simple cell complex $D(KS^k)\times \mathrm{Lk}(\sigma,K)$ to $K$ along $KS^k\times \mathrm{Lk}(\sigma,K)$.
\end{enumerate}
\end{definition}
In particular, if $\Tilde{K}$ denotes the cell complex obtained by the cellular surgery on $K$ then \[\Tilde{K}:= \bigg(K\setminus  KS^k\times D(\mathrm{Lk}(\sigma,K))\bigg) \underset{KS^k\times \mathrm{Lk}(\sigma,K)}{\bigcup}\bigg(D(KS^k)\times \mathrm{Lk}(\sigma,K) \bigg).\]

Let $K$ be a simple cell complex with free $\Z_2$-action such that $S^k$
embeds in $K$ as an $\Z_2$-equivariant subcomplex. 
Assume that, for any $k$-simplices $\sigma, \sigma'\in KS^k$ we have $\mathrm{Lk}(\sigma,K)\cong\mathrm{Lk}(\sigma',K)\cong S^{n-k-1}$ such that the quotient of $\mathrm{Lk}(\sigma,K)$ by $\Z_2$-action is again a cell complex.
With these assumptions, we are ready to define the projective version of a cellular surgery on the quotient of $K$ by the $\Z_2$-action.
\begin{definition}
Let $\mathbb{P}KS^k$ and $\overline{K}$ be the quotients of $KS^k$ and $K$ by the $\Z_2$-action, respectively.
The index $k$ projective cellular surgery on $\overline{K}$ along $\mathbb{P}KS^k$ is a cell complex $\Tilde{\overline{K}}$ defined as 
\[\Tilde{\overline{K}}:=\bigg(\overline{K}\setminus KS^k\times_{\Z_2} D(\mathrm{Lk}(\sigma,K))\bigg) \underset{KS^k\times_{\Z_2} \mathrm{Lk}(\sigma,K)}{\bigcup}\bigg(D(KS^k)\times_{\Z_2} \mathrm{Lk}(\sigma,K) \bigg),\]
where $KS^k\times_{\Z_2} D(\mathrm{Lk}$ denotes the quotient of $KS^k\times D(\mathrm{Lk}$ by diagonal $\Z_2$-action. Similarly, $KS^k\times_{\Z_2} \mathrm{Lk}(\sigma,K)$ and $D(KS^k)\times_{\Z_2} \mathrm{Lk}(\sigma,K)$ are defined.
\end{definition}

Let $CA_{m-1}$ be the Coxeter complex corresponding to the braid arrangement $\B_{m}$. 
Let $X\in \mathrm{Min}(\B_{m})$. Recall that $X$ can be represented by the partition of $[m]$ with at most one block of size greater equal $2$. Let $X=J-i_{1}-i_{2}-\dots- i_{k}$. 
Consider the subcollection 
\[\B_{X}=\{H\in \B_{m} : X\subset H\}\] of $\B_{m}$. 
It is easy to see that the following isomorphism  \[\B_{X}=\{H_{ij}\in \B_{m} : \{i,j\}\subset J\}\cong \B_{|J|}.\] 
Let $\sigma \in S_{X}$ be a cell such that $\dim(\sigma)=\dim(S_{X})$. 
From the above discussion, it is clear that $\mathrm{Lk}(\sigma, CA_{m-1})\cong CA_{|J|-1}$.
\begin{definition}\label{cs}
Let $X\in \mathrm{Min}(\B_{m})$.
Cellular surgery on $CA_{m-1}$ along $S_{X}$ is defined as
\begin{enumerate}
\item Truncate all cells which are adjacent to $S_{X}$.
\item Note that the boundary complex of the truncated part around $S_{X}$ is $X\times CA_{|J|-1}$. Let $D(S_{X})$ be the cellular disc whose boundary is $S_{X}$. Attach the complex $D(S_{X})\times CA_{|J|-1}$ along the boundary $S_{X}\times CA_{|J|-1}$. 
\end{enumerate}

\end{definition}

Similarly, we can define a cellular surgery on the projective Coxeter complex by replacing $S_{X}$ and $CA_{m-1}$ by $\mathbb{P}S_{X}$ and $\mathbb{P}CA_{m-1}$ respectively in the \Cref{cs}. Note that after truncating cells adjacent to $\mathbb{P}S_{X}$, the boundary of the truncated part will be $S_{X}\times_{O(1)} CA_{|J|-1}$. Accordingly, attach the $D(S_{X})\times_{O(1)} CA_{|J|-1}$ to the truncated complex.  

\begin{remark} We have the following observations.
\begin{enumerate}
\item It is easy to see that truncation of all cells adjacent to $S_{X}$ in $CA_{m-1}$ is an equivalent operation to removing $S^{m-|J|-1}\times D^{|J|-1}$ tubular neighbourhood of $S_{X}$, since $S_{X}\cong S^{m-|J|-1}$ and $ D^{|J|-1}$ is the $(|J|-1)$-dimensional disc. In step $2$ of the above definition, we attach $D^{m-|J|}\times S^{|J|-2}$ since, $CA_{|J|-1}\cong S^{|J|-2}$. Therefore, the \Cref{cs} is a cellular analogue of the surgery on manifolds.
    
\item If $dim(X)=0$ then the cellular surgery on $CA_{m-1}$ along $S_{X}$ gives the cell complex homeomorphic to $S^{1}\times S^{m-3}$. On the other hand, the cellular surgery on $\mathbb{P}CA_{m-1}$ along $\mathbb{P}S_{X}$ gives a cell complex which is homeomorphic to $\R P^{m-2}\#\R P^{m-2}$. 

\end{enumerate}

\end{remark}

\begin{example} \normalfont{
Let $X=123$ be an element of the minimal building set $\mathrm{Min}(\B_{4})$. Note that $S_{123}=S^0=\{123, \overline{123}\}$ is a $0$-dimensional sphere.  We can see that the subarrangement $\B_{X}$ is isomorphic to the braid arrangement $\B_{3}$. Note that there are $6$ triangles that are adjacent to $123$. Therefore, if we truncate all these triangles, the boundary of the truncated part will be a hexagonal circle, $CA_{2}$ (see the red hexagonal circle in \Cref{csoncc}). 
Similarly, truncating cells adjacent to $\overline{123}$ we get another hexagonal circle. 
Therefore, truncating cells adjacent to $S_{123}$ creates the disjoint union of two hexagonal circles as the boundary of the truncated part. 
Note that this boundary is isomorphic to the complex $S_{123}\times CA_{2}$. 
Let $D(S_{123})$ be the cellular disc with the boundary $S_{123}$.
In the next step of cellular surgery along $S_{123}$ we have to attach a hexagonal cylinder $D(S_{123})\times CA_{2}$, to the truncated complex along with the boundary complex  $S_{123}\times CA_{2}$ of the truncated part in the step-$1$ .
Now it is easy to see that the complex obtained after the cellular surgery is the torus cellulated by $18$ squares and $12$ triangles.}

\begin{figure}
$\includegraphics[scale=0.35]{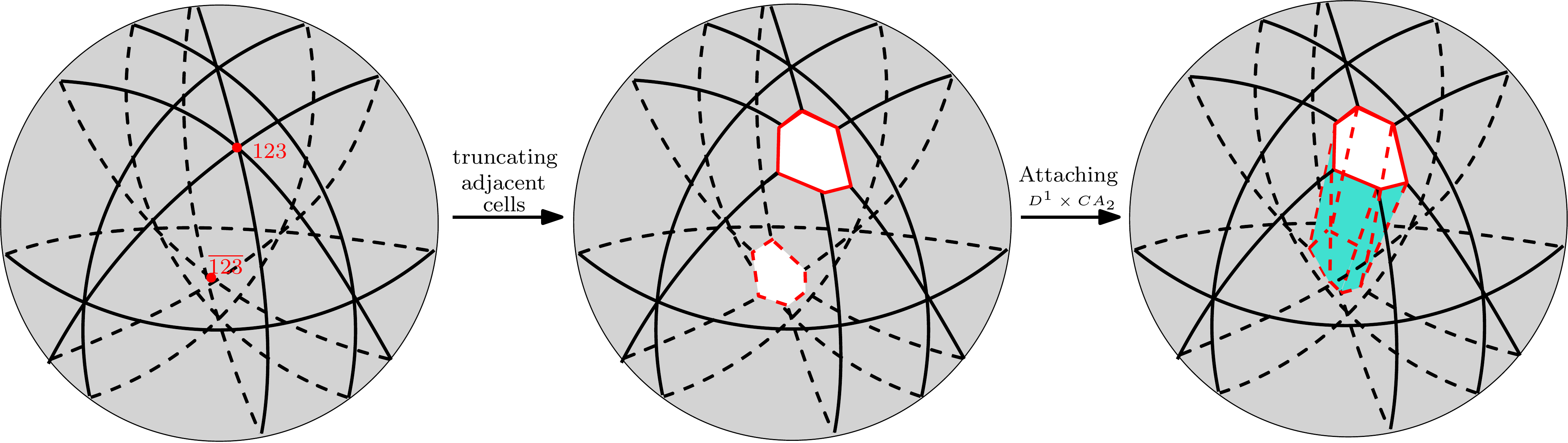}$
\vspace{.3cm}
\caption{Cellular surgery on $CA_{3}$ along $S_{123}$.}
\label{csoncc}
\end{figure}
\end{example}\

\begin{example}
\normalfont{
Let $X\in \mathrm{Min}(\B_{4})$ such that it is represented by an unordered partition $123-4$. Without loss of generality we can omit the singletons and write $X=123$. Consider the $0$-dimensional projective Coxeter complex $\mathbb{P}S_{123}$ in $\mathbb{P}X$. Similarly, as in the previous example we have $\mathbb{P}\B_{\mathbb{P}X}\cong \B_{3}$. Now truncating cells of $\mathbb{P}CA_{3}$ adjacent to $\mathbb{P}S_{123}$ gives boundary of truncated part to be $S_{123}\times_{O(1)}CA_{2}$, a hexagonal circle.  Note that the boundary $\partial(D(S_{123})\times_{O(1)} CA_{2})=S_{123}\times_{O(1)}CA_{2}$.
Now in the next step we attach $D(S_{123})\times_{O(1)} CA_{2}$ to $S_{123}\times_{O(1)CA_{2}}$. Note that $D(S_{123})\times_{O(1)} CA_{2}$ is a cell complex homeomorphic to the Mobius band. 
Now it is easy to see that the resulting complex after the projective cellular surgery is cellulated by $6$ triangles and $9$ squares (see \Cref{csonpcs}).}
\begin{figure}
$\includegraphics[scale=0.37]{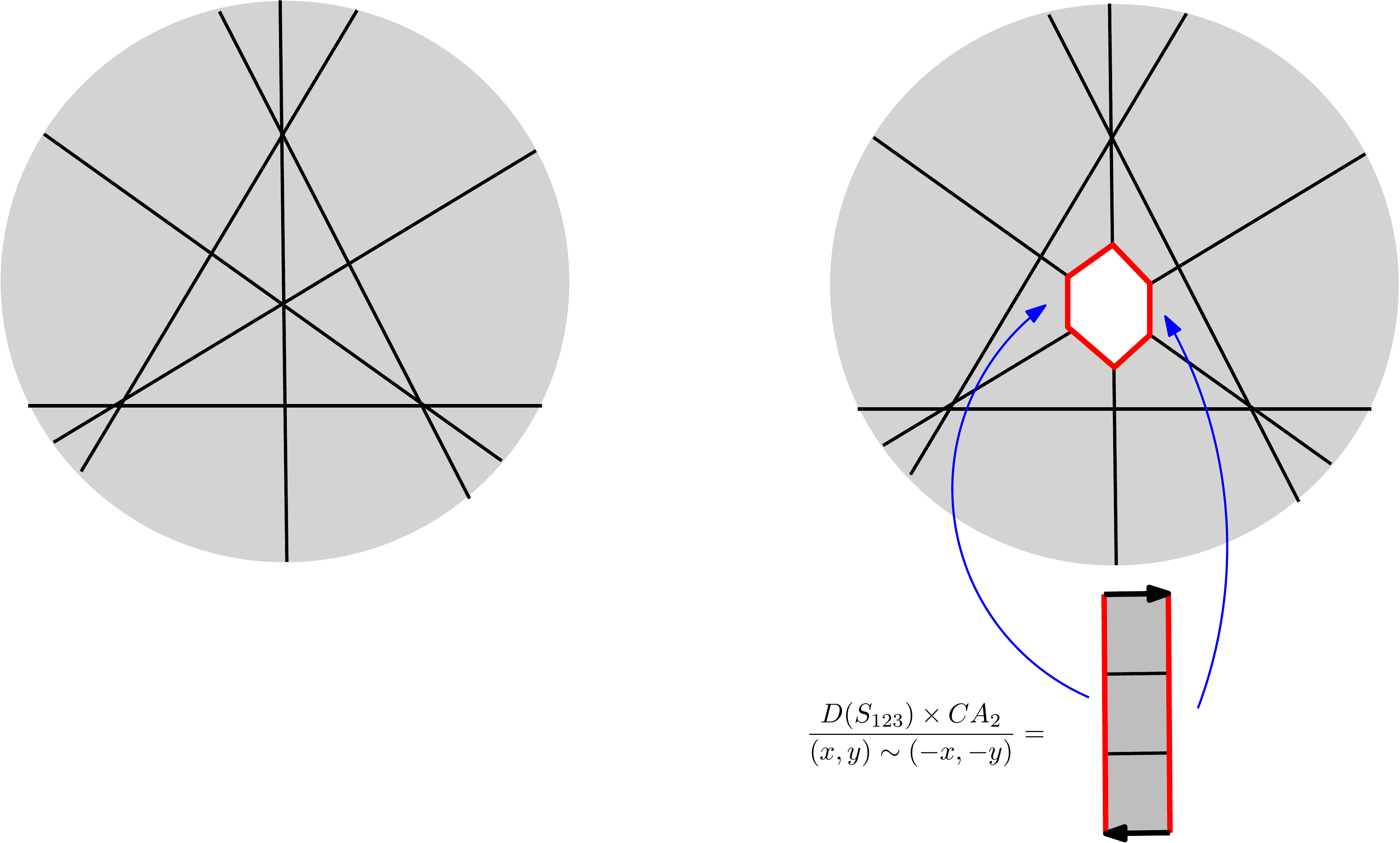}$
\caption{Cellular surgery on $\mathbb{P}CA_{3}$ along $\mathbb{P}S_{123}$.}
\label{csonpcs}
\end{figure} 
\end{example}

Let $\la m\ra=G_{1}\preceq G_{2}\preceq \dots \preceq G_{r}=G$ be a saturated chain of genetic codes such that $G_{i+1}$ covers $G_{i}$ for $1\leq i \leq r-1$ and  $S_{m}(G)=\{m\} \displaystyle \cup_{i=1}^{r-1} J_{i}$. 
Note that $m\notin J_{i}^c$. Therefore, $J_{i}^c$'s are short subsets with respect to the genetic code $\la m\ra$. Note that each $J_{i}^c$ represents the partition $J_{i}^c-j_{1}-j_{2}-\dots- j_{k}$ of $[m]$. 
Now it follows from the \Cref{min} that $\{J_{1}^c,\dots,J_{r-1}^c\}\subseteq \mathrm{Min}(\B_{m-1})$.
Consider the collections \[\mathcal{G}_{G}=\bigg\{S_{J_{1}^c}, S_{J_{2}^c},\dots,S_{J_{r-1}^c}\bigg\}\] and \[\mathbb{P}\mathcal{G}_{G}=\bigg\{\mathbb{P}S_{J_{1}^c}, \mathbb{P}S_{J_{2}^c},\dots,\mathbb{P}S_{J_{r-1}^c}\bigg\}.\]

\begin{theorem}\label{main thm}
Let $G$ be the genetic code of a length vector $\alpha$. Then
the iterated cellular surgery on $CA_{m-2}$ (respectively on $\mathbb{P}CA_{m-2}$) along the elements of $\mathcal{G}_{G}$ (respectively $\mathbb{P}\mathcal{G}_{G}$) produces the cell complex $\tilde{K}_{\alpha}$ (respectively $\tilde{\overline{K}}_{\alpha}$) homotopy equivalent to $\kalpha$ (respectively $\kbalpha$).
\end{theorem}
\begin{proof}
Following the inductive argument, it is enough to prove the theorem for a saturated chain of length $1$. Let $G\preceq G'$ be a saturated chain of length $1$. 
It follows from the \Cref{crgc} that, $S_{m}({G'})=S_{m}(G)\cup J$ for some $J\subset [m]$. 
Since $J^c$ is the maximal short subset (i.e., adding an extra element in $J^c$ makes it into long),
the subcomplex $S_{J^{c}}$ of $\mathrm{K}_G$ is isomorphic to the Coxeter complex $CA_{|J|-1}$ of dimension $|J|-2$.
Note that $J$ is short subset with respect to the genetic code $G'$. 
We also have $G'=\la G,J\ra$. Since $J$ is maximal short subset the subcomplex $S_J$ of $\mathrm{K}_{G'}$ represents the Coxeter complex $CA_{m-2-|J|}$. 
Now we see that the $\mathrm{Lk}(\sigma, \mathrm{K}_G)$ is isomorphic to the Coxeter complex for $\sigma\in S_{J^c}$ with $\dim(\sigma)=|J|-1$.
Recall that $\mathrm{M}_{G}$ is a PL-manifold. 
Therefore, $\mathrm{Lk}(\sigma, \mathrm{K}_G)\cong S^{n-|J|-2}$ if $\dim(\sigma)=|J|-1$.
The cell structure on $S^{n-|J|-2}$ is induced by the collection 
\[\{N_{i,j} : \{i,j\}\subset J^c\}.\]
Note that the above collection is isomorphic to the braid arrangement $\mathcal{B}_{m-|J|}$.
Therefore, $\mathrm{Lk}(\sigma, \mathrm{K}_G)\cong CA_{m-2-|J|}$.
Let $\Tilde{\mathrm{K}}_{G}$ be the complex obtained by the index $|J|-1$  cellular surgery on $\mathrm{K}_{G}$ along $S_{J^{c}}$. Then
\[\Tilde{\mathrm{K}}_{G}= \bigg(\mathrm{K}_{G}\setminus S_{J^{c}}\times D(CA_{m-2-|J|})\bigg) \underset{S_{J^c}\times CA_{m-2-|J|}}{\bigcup}\bigg(D(S_{J^c})\times CA_{m-2-|J|}\bigg).\]
Now if we collapse $D(S_{J^c})\times CA_{m-2-|J|}$ onto $CA_{m-2-|J|}$, $\Tilde{\mathrm{K}}_{G}$ becomes homotopy equivalent the complex $(\mathrm{K}_{G}\setminus S_{J^c})\cup S_{J}$. 
It follows from \Cref{Hausmann} that $\Tilde{\mathrm{K}}_{G}\cong \mathrm{M}_{G'}$.
Note that collapsing $D(S_{J^c})\times CA_{m-2-|J|}$ onto $CA_{m-2-|J|}$ doesn't change the homeomorphism type of $\Tilde{\mathrm{K}}_{G}$.
Therefore, $(\mathrm{K}_{G}\setminus S_{J^c})\cup S_{J}\cong\mathrm{M}_{G'}$.
Now it follows from \Cref{shortalpha} that the cell complex $(\mathrm{K}_{G}\setminus S_{J^c})\cup S_{J}$ is induced from the submanifold arrangement $\aalpha$. 
Therefore, $(\mathrm{K}_{G}\setminus S_{J^c})\cup S_{J}=\mathrm{K}_{G'}$.

Let $\mathbb{P}S_{J^c}$ be the projective Coxeter complex $\mathbb{P}CA_{|J|-1}$ in $\overline{\mathrm{K}}_{G}$ represented by a partition $J^c$ of $[m]$ and let $\mathbb{P}S_{J}$ be the subcomplex of $\Tilde{\overline{\mathrm{K}}}_{G'}$ isomorphic to the projective Coxeter complex $\mathbb{P}CA_{m-2-|J|}$.
The index $|J|-1$ projective cellular surgery on $\overline{\mathrm{K}}_{G}$ along $\mathbb{P}S_{J^c}$ gives \[\Tilde{\overline{\mathrm{K}}}_{G}= \bigg(\overline{\mathrm{K}}_{G}\setminus S_{J^c}\times_{O(1)} D(CA_{m-2-|J|})\bigg)\underset{S_{J^c}\times_{O(1)} CA_{m-2-|J|} }{\bigcup}\bigg(D(S_{J^c})\times_{O(1)} CA_{m-2-|J|}\bigg).\] 
Note that $S_{J^{c}}\times_{O(1)} D(CA_{m-2-|J|})$ and $D(S_{J^c})\times_{O(1)} CA_{m-2-|J|}$ are the total spaces of disc bundles over $\mathbb{P}S_{J^c}$ and $\mathbb{P}S_{J}$, respectively. 
Therefore, $J^{c}\times_{O(1)} D(CA_{m-2-|J|})$ and $D(J^c)\times_{O(1)} CA_{m-2-|J|}$ are homotopy equivalent to $\mathbb{P}S_{J^c}$ and $\mathbb{P}CA_{m-2-|J|}$, respectively. 
Therefore, $\Tilde{\overline{\mathrm{K}}}_{G}$ is homotopy equivalent to the complex $(\overline{\mathrm{K}}_{G}\setminus \mathbb{P}S_{J^c})\cup \mathbb{P}S_{J}$. 
Now the theorem follows from similar arguments as did for the cellular surgery.
\end{proof}

Since the projective cellular surgery along zero dimensional subspaces coincides with the blow-up, the following result is straightforward.
\begin{corollary}
Let $\la\{k,m\}\ra$ be the genetic code of $\alpha$. Then $\kbalpha$ is obtained from the  $\mathbb{P}CA_{m-2}$ by an iterated blow-up along the subspaces $\{1,2,\dots,\hat{i},\dots,m\}$ for $: 1\leq i\leq k$.
\end{corollary}

Now we characterize planar polygon spaces that are $\overline{\M}_{0}^{m}(\R)$. 

\begin{proposition}
Let $\alpha=(\alpha_1,\dots,\alpha_m)$
be a generic length vector. Then there is a homeomorphism $\overline{\M}_{0}^{m}(\R)\cong \kbalpha$ if and only if the genetic code of $\alpha$ is $\la\{45\}\ra$.
\end{proposition}
\begin{proof}
Devadoss \cite{Devadoss} showed that $\overline{\M}_{0}^{m}(\R)\cong \mbalpha$ is tiled by $(m-1)!/2$ copies of the associahedron of dmension $m-3$. 
Recall that the number of facets of this associahedron is $\binom{m-1}{2}-1$. 
Any top dimensional cell of $\kbalpha$ has at most $m$ many facets. 
Observe that $m=\binom{m-1}{2}-1$ if and only if $m=5$.
In this case the top dimensional cell (i.e. $2$-dimensional) of $\kbalpha$ has $5$-facets and it is isomorphic to a pentagon.
Recall that $\overline{\M}_{0}^{m}(\R)\cong \kbalpha$ is the connected sum of $5$ copies of $\mathbb{R}P^2$. 
Since $\overline{\mathrm{K}}_{\la\{4,5\}\ra}$ is tiled by $12$ pentagons and homeomorphic to the connected sum of $5$ copies of $\mathbb{R}P^2$, $\overline{\M}_{0}^{m}(\R)\cong \kbalpha$ for the genetic code $\la\{4,5\}\ra$. 
\end{proof}

We now illustrate the idea of the \Cref{main thm} through the following example.

\begin{example}\normalfont{
Consider the saturated chain of genetic codes $\la 5\ra\preceq \la15\ra\preceq \la25\ra \preceq \la125\ra.$
Note that $\mathcal{G}_{\la125\ra}=\{S_{234}, S_{134}, S_{34}\}$. 
Now we explain how to obtain the cell complex $\mathrm{K}_{\la125\ra}$ (resp. $\overline{\mathrm{K}}_{\la125\ra}$) by performing the cellular surgery on $CA_{3}$ (resp. $\mathbb{P}CA_{3})$ along $\mathcal{G}_{\la125\ra}$ (resp. $\mathbb{P}\mathcal{G}_{\la125\ra}$).
We start with performing surgery on $CA_{3}$ along $S_{234}$. Then we get the complex $\tilde{K}_{15}$ isomorphic to the torus. Note that, if we collapse the hexagonal cylinder onto one of its boundary components we get the complex again isomorphic to the torus. It is easy to see that this complex is isomorphic to the complex $\mathrm{K}_{\la15\ra}$. Later we follow the same process for $S_{134}$ and get the complex $\mathrm{K_{\la25\ra}}$. Now we need to do the surgery along $S_{34}$. Note that $S_{34}$ represents the hexagonal circle in $K_{\la25\ra}$. In this case, the first step is to truncate all the cells adjacent to $S_{34}$. After truncating adjacent cells we get the two disjoint complexes, each of them is isomorphic to the complex obtained from torus removing the hexagonal disc. In the second step, we attach the two disjoint unions of the hexagonal disc to the hexagonal boundary of each complex obtained in the previous step. Then we get the complex isomorphic to the disjoint union of two tori. Note that, if we collapse the attached hexagonal disc of to the point then again the resulting complex is isomorphic to the disjoint union of the torus which is exactly the complex $\mathrm{K}_{125}$. (see \Cref{itcs})

At every step of the iterated cellular surgery on $CA_{3}$, we can take the quotients by antipodal action and get the cellular surgery on $\mathbb{P}CA_{3}$. In particular, at the last step, we get the complex isomorphic to $\overline{\mathrm{K}}_{125}$, the torus.

The following arrows summarize the above process. 
\[CA_{3}\stackrel{234}{\longrightarrow}\tilde{K}_{15}\stackrel{h.e.}{\longrightarrow}\mathrm{K}_{\la15\ra}\stackrel{134}{\longrightarrow}\tilde{K}_{\la25\ra}\stackrel{h.e.}{\longrightarrow}\mathrm{K}_{\la25\ra}\stackrel{34}{\longrightarrow}\tilde{K}_{\la125\ra}\stackrel{h.e.}{\longrightarrow}\mathrm{K}_{\la125\ra}.\]

\[\mathbb{P}CA_{3}\stackrel{234}{\longrightarrow}\tilde{\overline{K}}_{15}\stackrel{h.e.}{\longrightarrow}\mathrm{\overline{K}}_{\la15\ra}\stackrel{134}{\longrightarrow}\tilde{\overline{K}}_{\la25\ra}\stackrel{h.e.}{\longrightarrow}\mathrm{\overline{K}}_{\la25\ra}\stackrel{34}{\longrightarrow}\tilde{\overline{K}}_{\la125\ra}\stackrel{h.e.}{\longrightarrow}\mathrm{\overline{K}}_{\la125\ra}.\]
}


\begin{figure}
$\includegraphics[scale=0.766]{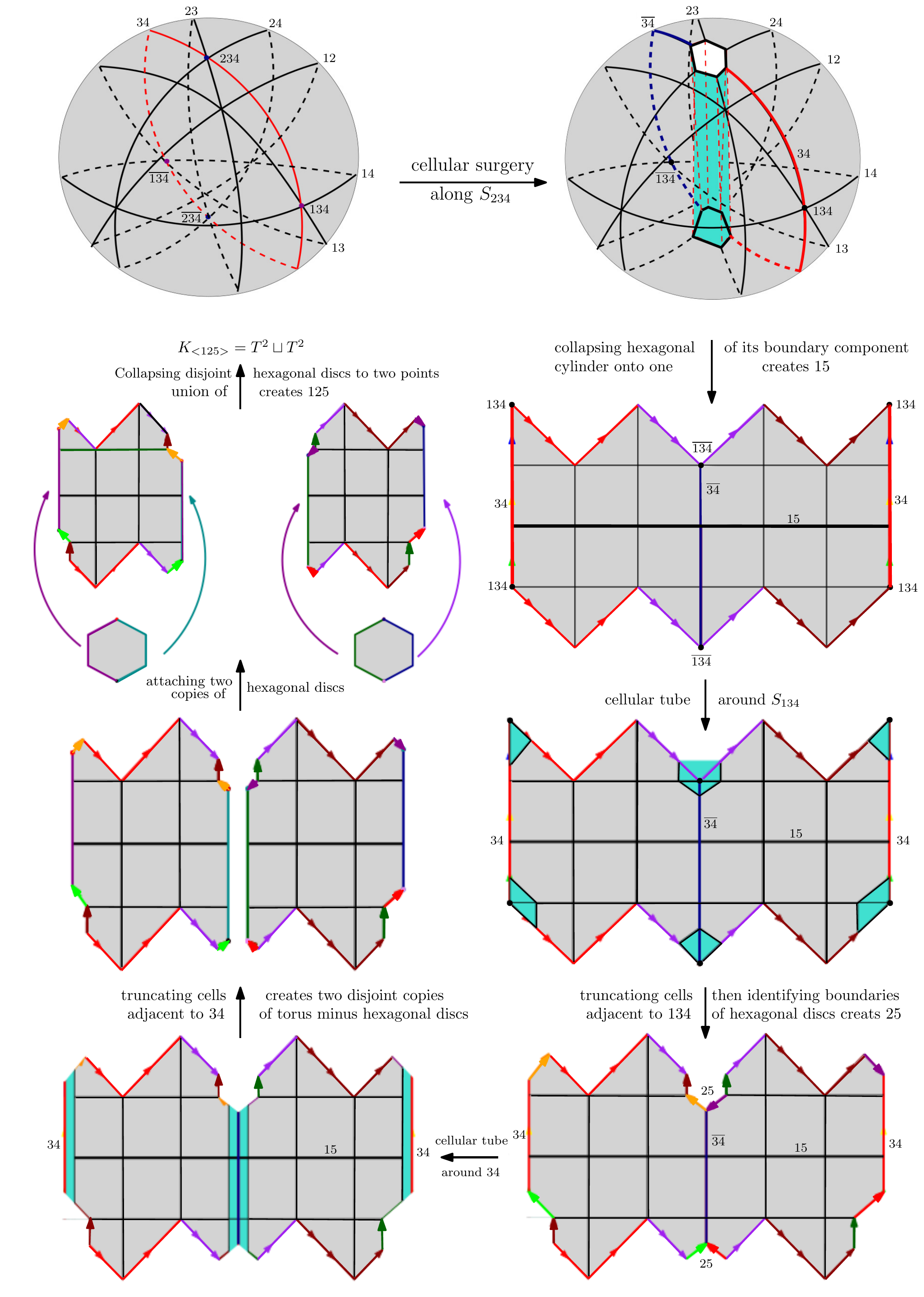}$
\caption{Iterated cellular surgery on $CA_{3}$ along $\mathcal{G}_{\la125\ra}$}
\label{itcs}
\end{figure} 
\end{example}



\noindent \textbf{Acknowledgements:}
The authors would like to thank the anonymous referee for the careful reading and important suggestions to improve the exposition of this article. The authors also thank Anurag Singh for useful discussions related to saturated chains of genetic codes.

\bibliographystyle{plain} 
\bibliography{references}

\end{document}